\documentclass[11pt]{article}
\usepackage[utf8]{inputenc}  
\usepackage[T1]{fontenc} 
\usepackage{lmodern}
\usepackage{amsmath} 
\usepackage{amssymb}
\usepackage{amsthm} 
\usepackage{mathrsfs}
\usepackage{graphicx}
\usepackage{ccaption}
\usepackage{hyperref}
\usepackage{listings}
\usepackage{enumitem}
\usepackage{tikz}
\usepackage{bm}
\usepackage{float}
\usepackage{pgfplots}
\usetikzlibrary{arrows}
\pgfplotsset{compat=newest}
\usetikzlibrary{shapes}
\usetikzlibrary{patterns}
\usepackage{stmaryrd}
\usepackage{authblk}
\usepackage[pagewise]{lineno}
\usetikzlibrary{arrows.meta,decorations.markings}
\usetikzlibrary{shapes.misc, positioning}
\usetikzlibrary{arrows}

\usepackage[left=2.5cm,right=2.5cm,top=3cm,bottom=3cm]{geometry}

\theoremstyle{definition}
\newtheorem{defi}{Definition}
\newtheorem{theorem}{Theorem}
\newtheorem{prop}{Proposition}
\newtheorem{coro}{Corollary}

\newtheorem{rem}{Remark}
\newcommand{\R}{\mathbb{R}}
\newcommand{\N}{\mathbb{N}}
\newcommand{\Z}{\mathbb{Z}}
\newcommand{\C}{\mathbb{C}}
\newcommand{\dd}{\mathrm{d}}
\newcommand{\F}{\mathcal{F}}
\newcommand{\eps}{\varepsilon}
\newcommand{\loc}{\text{loc}}

\newcommand{\ub}{\bm{u}}
\newcommand{\di}{\text{div}}
\newcommand{\curl}{\text{curl}}

\newcommand{\sib}{\bm{\sigma}}
\newcommand{\Xb}{\bm{X}}
\newcommand{\xb}{\bm{x}}
\newcommand{\Yb}{\bm{Y}}

\newcommand{\fb}{\bm{f}}
\newcommand{\bb}{\bm{b}}

\newcommand{\hau}{\tH^{(1)}}

\newcommand{\crit}{\text{crit}}
\newcommand{\Real}{\text{Real}}
\newcommand{\Imag}{\text{Imag}}

    \pgfplotsset{
        colormap={parula}{
            rgb255=(53,42,135)
            rgb255=(15,92,221)
            rgb255=(18,125,216)
            rgb255=(7,156,207)
            rgb255=(21,177,180)
            rgb255=(89,189,140)
            rgb255=(165,190,107)
            rgb255=(225,185,82)
            rgb255=(252,206,46)
            rgb255=(249,251,14)
        },
    }
\everymath{\displaystyle}   
\newcommand\tsum{\textstyle\sum\nolimits}

\def\wt{\widetilde}

\def\top{\text{top}}
\def\bot{\text{bot}}
\def\app{\text{app}}

\def\cO{\mathcal{O}}
\def\tH{\text{H}}

\title{Lamb modes and Born approximation for small shape defects inversion in elastic plates}
\author[1]{\'Eric Bonnetier}
\author[2,*]{Angèle Niclas}
\author[3]{Laurent Seppecher}
\affil[1]{Institut Fourier, Université Grenoble Alpes, France}
\affil[2]{CMAP, École Polytechnique, France}
\affil[3]{Institut Camille Jordan, \'Ecole Centrale Lyon, France}
\affil[*]{Corresponding author: angele.niclas@polytechnique.edu}
\date{}

\begin{document}
\maketitle
\begin{abstract}
The aim of this work is to present theoretical tools to study wave propagation in elastic waveguides and perform multi-frequency scattering inversion to reconstruct small shape defects in a 2D and 3D elastic plate. Given surface multi-frequency wavefield measurements, we use a Born approximation to reconstruct localized defect in the geometry of the plate. To justify this approximation, we introduce a rigorous framework to study the propagation of elastic wavefield generated by arbitrary sources. By studying the decreasing rate of the series of inhomogeneous Lamb mode, we prove the well-posedness of the PDE that model elastic wave propagation in 2D and 3D planar waveguides. We also characterize the critical frequencies for which the Lamb decomposition is not valid. Using these results, we generalize the shape reconstruction method already developed for acoustic waveguide to 2D elastic waveguides and provide a stable reconstruction method based on a mode-by-mode spacial Fourier inversion given by the scattered~field.
\end{abstract}

\section{Introduction}
This work is devoted to the reconstruction of small shape defects in a waveguide
using multi-frequency scattering data. It is an extension of the method exposed in \cite{bonnetier1}
int the case of acoustic waveguide to the case of elastic plates.
If the scalar Helmholtz case is relevant to the non destructive testing of pipes or optical fibers 
(see \cite{kharrat1}), applications in the elastic case concern the monitoring of structural parts,
airplane, ship, offshore wind energy plants or bridges for instance (see \cite{willberg1}).

The main common point between acoustic and elastic waveguides is the existence of a 
modal decomposition of the wavefield in a sum of explicit guided modes. 
The acoustic modes form an orthonormal basis, a property not satisfied by their elastic counter-parts,
called Lamb modes.
Several authors have looked into this feature. The books \cite{royer1,achenbach1} provide analytic expressions of Lamb modes as well as dispersion relations for their wavenumbers. 
In~\cite{maupin1,pagneux1,pagneux2} a new formulation is introduced, the $\Xb/\Yb$ formulation,
under which the family of Lamb mode turns out to be complete~\cite{akian1,kirrmann1,besserer1}. 
The associated bi-orthogonality relations \cite{fraser1} thus allow the use of the Lamb basis to decompose any wavefield that propagates in an elastic waveguide as a sum of Lamb modes.

However, a rigorous mathematical framework is still missing 
to study the propagation of an elastic wavefield generated by an arbitrary source term
(see however \cite{baronian1, baronian2} in 2D). 
One main goal of the present article is to prove well-posedness of the system of PDE's,
that models 2D or 3D planar elastic waveguides with internal and boundary source terms. 
To this end, we adapt the strategy developed for acoustic waveguides in \cite{bonnetier1}, 
which differs from \cite{baronian1}. 
Under stronger assumptions on the regularity of the source terms than those in \cite{baronian1}, 
we present in Theorem \ref{3_solution2D} a constructive proof of existence and regularity of a wavefield propagating 
in a two dimensional elastic waveguide.

As it turns out, this result is not valid at some particular frequencies, 
which we call critical frequencies,
and that are characterized in the proof of Theorem \ref{3_solution2D}. 
In particular, we establish in Corollary \ref{3_caraccritic} that the critical frequencies, 
for which the Lamb family is no longer complete, coincide with the vanishing of the bi-orthogonality relation established by \cite{fraser1}. This result, up to our knowledge, has not been proven 
before and may help understanding the mathematical analysis of elastic waveguides. 

Concerning the study of wave propagation in three-dimensional plates, most of the work
that we are aware of consists in adapting the 2D framework to situations with radial or
axial symmetry (see for instance \cite{legrand1, royer1,achenbach2, wilcox1}). In \cite{treyssede1},
arbitrary source terms are considered, without mathematical justification however.
Introducing the Helmholtz-Hodge decomposition of the wavefield \cite{bhatia1}, we split the three dimensional system of elasticity into a system of two independent equations. 
One of them fits into the scalar wave framework developed in \cite{bonnetier1}, 
while the other can be rewritten using the $\Xb/\Yb$ formulation. 
This provides a full expression for the decomposition of the wavefield generated by arbitrary source terms in dimension 3, see Theorem \ref{3_solution3D}. 

Equipped with these results, we can generalize the shape reconstruction method presented in \cite{bonnetier1} to the case of elastic plates, so as to determine possible defects
(bumps or dips) in the geometry of a plate, from multi-frequency measurements.
We use the very same procedure as in the acoustic case~: after mapping the perturbed plate to a 
straight configuration, we simplify the resulting system of equations using the Born approximation. 
The scattered wavefield generated by a known incident wavefield in the original geometry,
gives rise in the straightened plate to a boundary source term, that depends on the shape defect. 
Using measurements of the scattered field on the surface of the plate at different frequencies,
we can reconstruct in a stable way the shape defect (provided the latter is small enough). 
Numerical reconstructions are presented in the last part of the article, which show
the efficiency of the method.  

The paper is organized as follows. In section 2, we study the forward source problem in a two dimensional waveguide and introduce all the tools needed to use Lamb waves as a modal
basis. In section 3, we generalize the results of section 2 to the forward source problem in three dimensional plates. Section 4 is devoted to the reconstruction of shape defects in two dimensional plates, generalizing the method presented in \cite{bonnetier1}. Finally, in section 5 we show
numerical illustrations of the propagation of waves in two and three dimensional plates 
as well as reconstructions of different shape defects.

%%-----------------------------------------------------------------------------

\section{Forward source problem in a regular 2D waveguide\label{3_section2D}}
%%-----------------------------------------------------------------------------

In this section, we present a complete study of the forward elastic source problem in a 
two-dimensional regular waveguide. We use the  $\Xb/\Yb$ formulation developed in \cite{pagneux1,pagneux2} which allows a modal decomposition of any elastic wavefield using Lamb modes. Most of the results presented here are already known, and can be found in \cite{pagneux2,royer1,achenbach1}. Our main contribution is to provide a rigorous
proof of well-posedness for the direct problem and of the fact that its solutions
can be represented in terms of Lamb modes (Theorem~\ref{3_solution2D}).
We also follow the suggestions in \cite{kirrmann1} to define the set of critical frequencies and critical wavenumbers in Definition \ref{3_deficrit}, and we prove in Corollary \ref{3_caraccritic} that it coincides with the set of frequencies for which the components $\Xb_n$ and $\Yb_n$
of the eigenmodes are orthogonal for some $n$.

%%-----------------------------------------------------------------------------
\subsection{Lamb modes and critical frequencies} \label{3_sec_2}
 %%-----------------------------------------------------------------------------

We consider a 2D infinite, straight, elastic waveguide $\Omega=\{(x,z) \in \R\times(-h,h)\}$ 
of width $2h > 0$. 
The displacement field is denoted by $\ub=(u,v)$ . Given a frequency $\omega \in \R$, 
and given $(\lambda,\mu)$ the Lam\'e parameters of the elastic waveguide, 
the wavefield $\ub$ satisfies
\begin{equation} \label{3_lamb1}
\nabla\cdot\sib(\ub)+\omega^2\ub=-\bm{f}\qquad \text{ in } \Omega, 
\end{equation}
where $\fb= (f_1,f_2)$ is a given source term, and where
the stress tensor $\sib(\ub)$ is defined by
\begin{equation}
\sib(\ub)=\left(\begin{array}{cc} (\lambda+2\mu)\partial_x u +\lambda\partial_z v & \mu \partial_z u+\mu \partial_x v \\ \mu \partial_z u+\mu \partial_x v  & \lambda\partial_x u +(\lambda+2\mu) \partial_z v \end{array}\right):=\left(\begin{array}{cc} s & t \\ t & r \end{array}\right). 
\end{equation}
In this work, we assume that a Neumann boundary condition is imposed on
both sides of the plate
\begin{equation} \label{3_neumann}\sib(\ub)\cdot \nu=\bb^\top \quad  \text{ on } \partial\Omega_\top, \qquad \sib(\ub)\cdot \nu=\bb^\bot \quad  \text{ on } \partial\Omega_\bot, \end{equation}
where $\bb^\top=(b_1^\top,b_2^\bot)$ and $\bb^\bot=(b_1^\bot,b_2^\bot)$ are 
given boundary source terms. 
This condition could easily be replaced by a Dirichlet or a Robin condition without much changes in the following analysis. The setting is represented in Figure \ref{3_guide2D}.

\begin{figure}[h]
\begin{center}
\begin{tikzpicture}
\draw (-5.5,0.7) -- (5,0.7);
\draw (-5.5,-0.7) -- (5,-0.7);
\draw [white,fill=gray!40] (0,0.2) circle (0.4);
\draw (0,0.2) node{$\fb$};
\draw (2.5,0) node{$\Omega$};
\draw (-5.5,0.7) node[left]{$h$}; 
\draw (-5.5,-0.7) node[left]{$-h$}; 
\draw [ultra thick][->] (-5,-0.3)--(-4.3,-0.3) node[above]{$e_x$}; 
\draw [ultra thick][->] (-5,-0.3)--(-5,0.4) node[right]{$e_z$}; 
\draw [white,fill=gray!40] (-3.5,0.6)--(-3.5,0.8)--(-2.5,0.8)--(-2.5,0.6)--(-3.5,0.6); 
\draw (-3,0.8) node[above]{$\bb^\top$}; 
\draw [white,fill=gray!40] (-1,-0.6)--(-1,-0.8)--(2,-0.8)--(2,-0.6)--(-1,-0.6); 
\draw (0.5,-0.8) node[below]{$\bb^\bot$}; 
\end{tikzpicture}
\end{center}
\caption{\label{3_guide2D} Parametrization of a two dimensional plate $\Omega$. Elastic wavefields are generated using an internal source term $\fb$, and boundary source terms $\bb^\top$ and $\bb^\bot$. }
\end{figure}
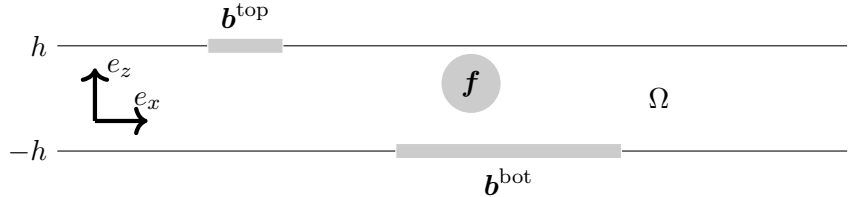

In~\cite{maupin1} this equation is analyzed in an operator form $\bm{Z}=\mathcal{L}(\bm{Z})$ where $\bm{Z}=(u,t,s,v)$. This idea was then adapted in \cite{pagneux1} to formalize the so-called $\Xb/\Yb$ formulation. We introduce the variables 
\begin{equation}\label{3_XY}
\Xb=(u,t), \qquad \Yb=(-s,v),
\end{equation}
with which the elasticity equation can be rewritten as follows:
%%-----------------------------------------------------------------------------
\begin{prop}\label{3_equivalenceXY}
The system \eqref{3_lamb1}, with the Neumann boundary conditions \eqref{3_neumann}, is equivalent to 
\begin{equation}\label{3_eqXY}
 \partial_x\left(\begin{array}{c} \bm{X} \\ \bm{Y} \end{array}\right)=
\mathcal{L}(\Xb,\Yb)
+\left(\begin{array}{c} 0 \\ -f_2 -b_2^\top \delta_{z=h}-b_2^\bot\delta_{z=-h} \\ f_1+ b_1^\top \delta_{z=h}+b_2^\top\delta_{z=-h} \\ 0 \end{array} \right) \quad \text{ in } \Omega,
\end{equation}
with the boundary condition $B_1(\bm{X})=B_2(\bm{Y})=0$, where $\mathcal{L}(\Xb,\Yb) = ( F(\bm{Y}) ; G(\bm{X}))$
and 
$F$, $G$, $B_1$ and $B_2$ are differential matrix operators defined by 
\begin{equation}\label{3_FG}
F=\left(\begin{array}{cc} -\displaystyle\frac{1}{\lambda+2\mu} & -\displaystyle\frac{\lambda}{\lambda+2\mu}\partial_z \\ \displaystyle\frac{\lambda}{\lambda+2\mu}\partial_z & -\omega^2-\displaystyle\frac{4\mu(\lambda+\mu)}{\lambda+2\mu}\partial^2_{zz} \end{array}\right), \qquad G=\left(\begin{array}{cc}
\omega^2 & \partial_z \\ -\partial_z & \displaystyle\frac{1}{\mu} 
\end{array}\right),
\end{equation}
\begin{equation}\label{3_B1B2}
B_1(\bm{X})=\bm{X}\cdot e_z, \qquad B_2(\bm{Y})=-\frac{\lambda}{\lambda+2\mu}\bm{Y}\cdot e_x +\frac{4\mu(\lambda+\mu)}{\lambda+2\mu}\partial_z\bm{Y}\cdot e_z. 
\end{equation}
\end{prop}
%%-----------------------------------------------------------------------------

The proof of this proposition follows the same steps as that presented in Appendix~A of~\cite{pagneux2}.
In this formulation, the operators $F$ and $G$ only depend on $z$, and are defined on one section of the waveguide, while derivatives with respect to $x$ only appear in the left-hand side of~\eqref{3_eqXY}. We consider the space 
\begin{equation}
H_0:=\left\{(\Xb,\Yb)\in (\text{H}^2(-h,h))^4 \,|\, B_1(\Xb)(\pm h)=B_2(\Yb)(\pm h)=0 \right\},
\end{equation}
and the operator 
\begin{equation}
\mathcal{L}: \begin{array}{rcl} H_0 & \rightarrow & (\text{L}^2(-h,h))^4 \\ 
(\Xb,\Yb) & \mapsto & (F(\Yb),G(\Xb)) \end{array} .
\end{equation}
Our goal is to diagonalize this operator and, to this end, we introduce the Lamb modes: 

%%-----------------------------------------------------------------------------
\begin{defi}\label{3_lambmode}
A Lamb mode $(\bm{X},\bm{Y})\in H_0$, associated to the wavenumber $k\in \C$, is a 
non-trivial solution of 
$\mathcal{L}(\Xb,\Yb)=ik(\Xb,\Yb)$.
\end{defi}
%%-----------------------------------------------------------------------------

The next Proposition provides the analytical expressions of these modes.
The proof can be found in~\cite{achenbach1,royer1}.

%%-----------------------------------------------------------------------------
\begin{prop}
The set of wavenumbers $k\in \C$ associated to Lamb modes is countable, and every
such  wavenumber $k$ satisfies the symmetric Rayleigh-Lamb equation 
\begin{equation} \label{3_disps} p^2=\frac{\omega^2}{\lambda+2\mu}-k^2, \qquad q^2=\frac{\omega^2}{\mu}-k^2, \qquad \left(q^2-k^2\right)^2=-4k^2pq\frac{\tan(ph)}{\tan(qh)},
\end{equation}
or the antisymmetric Rayleigh-Lamb equation
\begin{equation}
\label{3_dispa} p^2=\frac{\omega^2}{\lambda+2\mu}-k^2, \qquad q^2=\frac{\omega^2}{\mu}-k^2, \qquad \left(q^2-k^2\right)^2=-4k^2pq\frac{\tan(qh)}{\tan(ph)}.
\end{equation}
If $k$ satisfies \eqref{3_disps}, the associated Lamb mode is called symmetric and is proportional to \\ $(\Xb(z),\Yb(z))=$
\begin{equation}\label{3_lambs}
 \left(\begin{array}{c} u(z) \\ t(z) \\ -s(z) \\ v(z)\end{array}\right) =\left(\begin{array}{c} 
ik(q^2-k^2)\sin(qh)\cos(pz)-2ikpq\sin(ph)\cos(qz) \\ 
2ik\mu(q^2-k^2)p(-\sin(qh)\sin(pz)+\sin(ph)\sin(qz))\\
 (q^2-k^2)((\lambda+2\mu)k^2+\lambda p^2)\sin(qh)\cos(pz) -4\mu pq k^2\sin(ph)\cos(qz)\\
-p(q^2-k^2)\sin(qh)\sin(pz)-2k^2p\sin(ph)\sin(qz) \end{array}\right). 
\end{equation}
If $k$ satisfies \eqref{3_dispa}, the associated Lamb mode is called anti-symmetric and is proportional to $(\Xb(z),\Yb(z))=$
\begin{equation} \label{3_lamba}
 \left(\begin{array}{c} u(z) \\ t(z) \\ -s(z) \\ v(z)\end{array}\right) =\left(\begin{array}{c} 
ik(q^2-k^2)\cos(qh)\sin(pz)-2ikpq\cos(ph)\sin(qz) \\ 
2ik\mu(q^2-k^2)p(\cos(qh)\cos(pz)-\cos(ph)\cos(qz))\\
 (q^2-k^2)((\lambda+2\mu)k^2+\lambda p^2)\cos(qh)\sin(pz) -4\mu pq k^2\cos(ph)\sin(qz)\\
p(q^2-k^2)\cos(qh)\cos(pz)+2k^2p\cos(ph)\cos(qz) \end{array}\right). 
\end{equation}
\end{prop}
%%-----------------------------------------------------------------------------

\begin{rem}
We see on the above expressions that $p$ and $q$ are defined up to a multiplication by $-1$. 
However, since Lamb modes are defined up to a multiplicative constant, the choice of 
the sign of $p$ or $q$ does not change the associated value of $k$ or the associated Lamb mode. 
\end{rem}

We notice that if $k$ is a solution of the Rayleigh-Lamb equation then $-k$ and $\bar{k}$ are also solutions. Figure \ref{3_disp3D} depicts different wavenumbers $k$ where $\text{Real}(k)\geq 0$ and $\text{Imag}(k)\geq 0$, in terms of the frequency $\omega$. 

\begin{figure}[h]
\begin{center}
\input{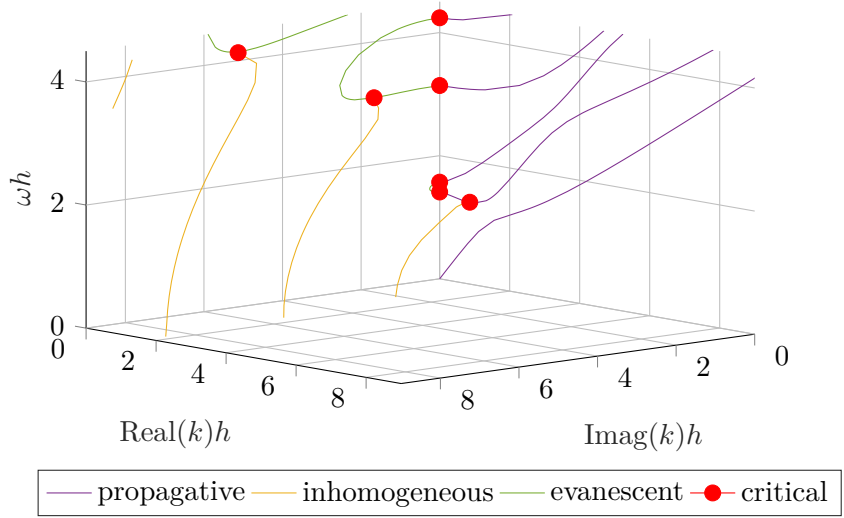}
\end{center}
\caption{\label{3_disp3D} Solutions of the symmetric Rayleigh-Lamb equation \eqref{3_lambs} in the space $\Imag(k)\geq 0$, $\Real(k)\geq 0$ with $\mu=0.25$ and $\lambda=0.31$. Solutions on the full space can be obtained by axial symmetries. Propagative, evanescent and inhomogeneous modes are represented by different colors. Critical points are represented by red dots.}
\end{figure}

We can distinguish three different types of modes (represented in different colors
in the above Figure) as in \cite{legrand1}:

%%-----------------------------------------------------------------------------
\begin{defi}
There are three types of Lamb modes:
\begin{itemize}
\item If $k\in \R$, the mode oscillates in the waveguide without energy decay and is called propagative. 
\item If $k\in i\R$, the mode decays exponentially to zero 
as $|x| \to \infty$, and is called evanescent. 
\item If $\Real(k)\neq 0$ and $\Imag(k)\neq 0$, the mode oscillates quickly toward zero and is called inhomogeneous. 
\end{itemize}
\end{defi}
%%-----------------------------------------------------------------------------

The completeness of Lamb modes depends on whether the frequency $\omega$
is critical as defined below:

%%-----------------------------------------------------------------------------
\begin{defi}\label{3_deficrit}
A frequency $\omega$ and a wavenumber $k$ are said to be critical if they satisfy $k=0$ 
or condition~\eqref{3_disps} and $\Gamma_S=0$ 
(resp. condition \eqref{3_dispa} and $\Gamma_A=0$)
 where 
\begin{eqnarray}
\label{3_crits}
\Gamma_S&=&
h(q^2-k^2)^2\sin(qh)^2+4k^2p^2\sin(ph)^2\\
&&
+(q^2-k^2)\sin(ph)\cos(ph)\sin(qh)^2\left(\frac{q^2-k^2}{p}-8p-\frac{2p}{k^2}-\frac{p}{q^2}\right),
\nonumber \\
\label{3_crita}
\Gamma_A &=&
h(q^2-k^2)^2\cos(qh)^2+4k^2p^2\cos(ph)^2\\
&&
-(q^2-k^2)\cos(ph)\sin(ph)\cos(qh)^2\left(\frac{q^2-k^2}{p}-8p-\frac{2p}{k^2}-\frac{p}{q^2}\right).
\nonumber 
\end{eqnarray}
We denote by $\omega_\crit$ the set of critical frequencies, and 
by $k_\crit$ the set of associated critical wavenumbers. 
\end{defi}
%%-----------------------------------------------------------------------------

Critical points solution to $k=0$ or to $\Gamma_S=0$, where $\Gamma_S$ is defined in \eqref{3_crits},
are represented in Figure \ref{3_disp3D}. We notice that critical points seem to be located at the junction of branches of modes of different types (see \cite{kirrmann1,royer1} for more details). Next, we introduce the functional space
\begin{equation}
H:=\text{H}^1(-h,h)\times \text{L}^2(-h,h)\times \text{L}^2(-h,h)\times \text{H}^1(-h,h).
\end{equation}
and state the following completeness result: 
%%-----------------------------------------------------------------------------
\begin{theorem}\label{3_th2}
At frequency $\omega$, Lamb modes form a complete set of functions in $H$ if and only if $\omega\notin \omega_\crit$.
\end{theorem}
%%-----------------------------------------------------------------------------

\begin{proof}
{\it Step 1~:}
It is shown in \cite{akian1} (see also~\cite{kirrmann1,besserer1}) that 
the operator $\mathcal{L}$ satisfies the following properties: 
\begin{itemize}
\item There exists a set of five rays in the complex plane
such that the angles between adjacent rays are less than $\pi/2$,
\item Sufficiently far from the origin, all the points on these rays lie in 
the resolvent set of $\mathcal{L}$, 
\item There exits $N\in \N$ such that the resolvent of $\mathcal{L}$ satisfies 
\end{itemize}
\begin{equation}
\Vert (\mathcal{L}-\ell I)^{-1})\Vert =\cO(|\ell|^N) \text{ as } |\ell|\rightarrow +\infty \text{ along each ray.}
\end{equation}
Invoking Theorem 6.2 in \cite{locker1}, one may then infer that the family of Lamb modes 
forms a complete set of functions if and only if for every associated wavenumber $k$,
\begin{equation}\label{3_eqcrit}
\text{Ker}(\mathcal{L}-ikI)=\text{Ker}(\mathcal{L}-ikI)^2.
\end{equation}

{\it Step 2~:}
The above condition is however implicit and does not allow an effective determination 
of the frequencies for which the Lamb modes form a complete set. 
In~\cite{kirrmann1}, a simpler condition than \eqref{3_eqcrit} is derived (although not proved)
with a reference to \cite{stange1}.
Our goal is to derive an equivalent condition, 
that only depends on the parameters of the problem. 
To this end, we generalize the approach in \cite{royer1} 
and in view of~\eqref{3_eqcrit}, we seek to characterize under which conditions 
generalized eigenvalues exist.
We present the argument in the case of a symmetric Lamb mode, the antisymmetric
situation can be handled in the same manner.

Assume that $(u_0,t_0,-s_0,v_0)$ is a symmetric Lamb mode 
(given by \eqref{3_lambs}) associated with a wavenumber $k \in \C$ so that~\eqref{3_eqcrit} is not satisfied. We look for $(u,t,-s,v)\in H_0$ that satisfies 
\begin{equation} \label{3_eq_gen_ef}
\mathcal{L}(u,t,-s,v)=ik\mathcal{L}(u,t,-s,v)+(u_0,t_0,-s_0,v_0).
\end{equation}
Defining 
\begin{equation}
f_1=-(\lambda+\mu) \partial_z u_0-2ik\mu v_0, \qquad f_2
=-2(\lambda+2\mu)iku_0-(\lambda+\mu)\partial_z v_0,
\end{equation}
we notice that $(u,v)$ satisfies the equation 
\begin{equation}
\left\{\begin{array}{cl}(\lambda+2\mu)\partial_{zz} v +(\lambda+\mu) ik\partial_z u -\mu q^2v=f_1 &\text{ in } (-h,h), \\ 
\mu \partial_{zz} u +(\lambda+\mu) ik\partial_z v -p^2(\lambda+2\mu) u =f_2 & \text{ in } (-h,h), \\
\partial_z u(\pm h)+ikv(\pm h)=-v_0(\pm h), & \\
\lambda ik u(\pm h)+(\lambda+2\mu) \partial_z v(\pm h)=-\lambda u_0(\pm h). & \end{array}\right. 
\end{equation}
We introduce the auxiliary functions
\begin{equation}
\phi=-\frac{\lambda+2\mu}{\omega^2}(ik u+\partial_z v), 
\qquad \psi=\frac{\mu}{\omega^2}(\partial_z u -ikv),
\end{equation}
which turn out to solve the following second order linear ODE's with constant coefficients
\begin{equation}
\partial_{zz} \phi(z)+p^2\phi(z)
=-2ik(q^2-k^2)\sin(qh)\cos(pz)+2ikpq\frac{\lambda+\mu}{\mu} \sin(ph)\cos(qz),
\end{equation}
\vspace{-6mm}
\begin{equation}
\partial_{zz} \psi(z) +q^2\psi(z)
=-(q^2-k^2)p\sin(qh)\frac{\lambda+\mu}{\lambda+2\mu}\sin(pz)+4k^2p\sin(ph)\sin(qz).
\end{equation}
The solutions of the above ODE's are explicit. Using the fact that $u=ik\phi-\partial_z \psi+f_2/\omega^2$ and $v=\partial_z\phi+ik\psi+f_1/\omega^2$, we find that
\begin{eqnarray}\label{3_ug}
u(z)&=&
A_1 ik\cos(pz)+A_2 ik\sin(pz)+A^\prime_1 q\sin(qz)-A^\prime_2 q\cos(qz)\\
&&
+\frac{k^2(q^2-k^2)}{p}\sin(qh)z\sin(pz)-2k^2p\sin(ph)z\sin(qz)
+\frac{2k^4\mu p}{q\omega^2}\sin(ph)\cos(qz) 
\nonumber\\
&&
+\left(q^2-k^2+\frac{k^2(q^2-k^2)(\lambda+2\mu)}{\omega^2}\right)\sin(qh)\cos(pz),
\nonumber
\\
\label{3_vg}
v(z)&=&
-A_1p\sin(pz)+A_2p\cos(pz)+A^\prime_1 ik\cos(qz)+A^\prime_2 ik\sin(qz)\\
&&
-ik(q^2-k^2)\sin(qh)z\cos(pz)-\frac{2k^2p}{q}\sin(ph)z\cos(qz) +\left(2ikp+\frac{2ik^3\mu p}{\omega^2}\right)\sin(ph)\sin(qz) 
\nonumber \\
&&
-\frac{ik^3(q^2-k^2)(\lambda+2\mu)}{\omega^2p}\sin(qh)\sin(pz),
\nonumber
\end{eqnarray}
for some $A_1,A_2,A^\prime_1, A^\prime_2 \in \C$. Expressing the boundary conditions, we obtain 
\begin{equation} \label{3_cb} M\left(\begin{array}{c} A_1 \\ A^\prime_2 \end{array}\right) = -\left(\begin{array}{c} b_1 \\ b_2\end{array}\right), \quad \text{with} \quad  \left(\begin{array}{cc} -2ikp\sin(ph) & (q^2-k^2)\sin(qh) \\
-\mu(q^2-k^2)\cos(ph) & 2ik\mu q \cos(qh) \end{array}\right),
\end{equation}
\vspace{-7mm}
\begin{multline*}
b_1=2k^2(q^2-k^2)h\cos(ph)\sin(qh)-\frac{2(q^2-k^2)k^2p}{q}h\cos(qh)\sin(ph)\\ +\sin(ph)\sin(qh)\left(-9pk^2-pq^2+\frac{4k^2pq^2}{\omega^2}+\frac{(q^2-k^2)(k^2-p^2)(\omega^2+(\lambda+2\mu)k^2)}{p\omega^2}\right), 
\end{multline*}
\vspace{-7mm}
\begin{multline*}
b_2=\left(4\mu ik^3p+\frac{ik\mu(q^2-k^2)^2}{p}\right)h\sin(ph)\sin(qh) + \cos(qh)\sin(ph)\left(4ik\mu qp-\frac{4ik^5p\mu^2}{q\omega^2}\right)\\+\cos(ph)\sin(qh)\left(-\frac{2ik\mu(q^2-k^2)}{\omega^2}(\omega^2+k^2(\lambda+2\mu))+\lambda ik(q^2-k^2)\right). 
\end{multline*}
Since $k$ satisfies the symmetric Rayleigh-Lamb equation \eqref{3_disps}, the determinant of the matrix $M$ vanishes.

If the first column of the matrix $M$ is non zero, then \eqref{3_cb} has a solution 
if and only if 
\begin{equation}
\left|\begin{array}{cc} 2ikp\sin(ph) & b_1 \\ \mu(q^2-k^2)\cos(ph) & b_2\end{array}\right|=0. 
\end{equation}
Computing this determinant and using the relation \eqref{3_disps} leads to 
\vspace{-3mm}
\begin{multline} \label{3_cond_c2}
-2k^2\mu\Bigg[h(q^2-k^2)^2\sin(qh)^2+4k^2p^2\sin(ph)^2\\+(q^2-k^2)\sin(ph)\cos(ph)\sin(qh)^2\left(\frac{q^2-k^2}{p}-8p-\frac{2p}{k^2}-\frac{p}{q^2}\right)\Bigg]=0,
\end{multline}
in other words, $\Gamma_S = 0$ and $\omega \in \omega_{\crit}$.

Assume now that $M_{11} = M_{21} = 0$. Then either $k = 0$ and $q^2\cos(ph) = 0$
(the condition $q = 0$ is excluded as it yields to a trivial eigenfunction 
$(u_0, t_0, -s_0, v_0)$) and the system takes the form
\begin{equation}
\left(\begin{array}{cc} 0 & q^2\sin(qh) \\ 0 & 0 \end{array}\right)
\left(\begin{array}{c} A_1 \\ B_2 \end{array}\right) 
=\left(\begin{array}{c} 2pq^2\sin(qh) \\ 0 \end{array}\right), 
\end{equation}
and has non trivial solutions (for instance $B_2=2p$ and $A_1=0$). 
Or $k \neq 0$, $p \in \pi \Z$ and $q^2 - k^2 = 0$, in which case 
$(b_1, b_2) = (0,0)$
and the system also has nontrivial solutions. In this latter case, one can check
that $\Gamma_S = 0$ as well.

Conversely, if~\eqref{3_cond_c2} holds, then using~\eqref{3_ug}-\eqref{3_vg} one can construct 
a solution $(u,t,-s,v)$ to~\eqref{3_eq_gen_ef}.
This shows that $\omega \in \omega_{\crit}$ if an only if \eqref{3_eqcrit} is
not satisfied, and concludes the proof of the Theorem.
\end{proof}

\begin{rem}
When $\omega\in \omega_\crit$, one needs to add generalized eigenmodes to
the Lamb modes to obtain a complete family~\cite{kirrmann1,akian1}.
Our proof can be useful if ones want to find the expression of such generalized modes: 
finding $(A_1,B_2)$ solution to \eqref{3_cb} and replacing it in \eqref{3_ug} and \eqref{3_vg} gives the expression of the generalized modes.
\end{rem}

\begin{rem}
If we derive equation \eqref{3_disps} (resp. \eqref{3_dispa}) with respect to $k$, we notice that 
$\partial_k \omega =0 $ if and only if \eqref{3_crits} (resp. \eqref{3_crita}) is satisfied. 
This shows that critical points are exactly located where $\partial_k \omega =0$. 
When $k \neq 0$, these points are called zero velocity group points (ZGV points) 
and have been extensively studied (see for instance \cite{balogun1}). 
\end{rem}

%%-----------------------------------------------------------------------------
\subsection{Solution of the 2D elasticity problem}
%%-----------------------------------------------------------------------------

In the rest of this section, we assume that $\omega \notin \omega_\crit$, so that Lamb modes 
form a complete family, however they do not necessarily yield an orthonormal basis. 
In order to identify the decomposition of a given function of $H$ on the Lamb basis, 
we split the set of wavenumbers $k$ in two parts: 

\begin{defi}
Let $\omega \notin \omega_\crit$. 
\begin{itemize}
\item We say that a Lamb mode with wavenumber $k$ is right-going if
$\Imag(k)>0$ or $\Imag(k)=0$ and $\partial_k\omega>0$,
\item We say that a Lamb mode is left-going if
$\Imag(k)<0$ or $\Imag(k)=0$ and $\partial_k\omega<0$. 
\end{itemize}
We index the right-going modes by $n\in \N^*$, and sort them by ascending order of imaginary part 
and descending order of real part. 
\end{defi} 

We illustrate this classification in Figure~\ref{3_rgmode}, where right and left-going wavenumbers
are represented at the frequency $\omega=1.37$. 

\begin{figure}[h]
\begin{center}
\input{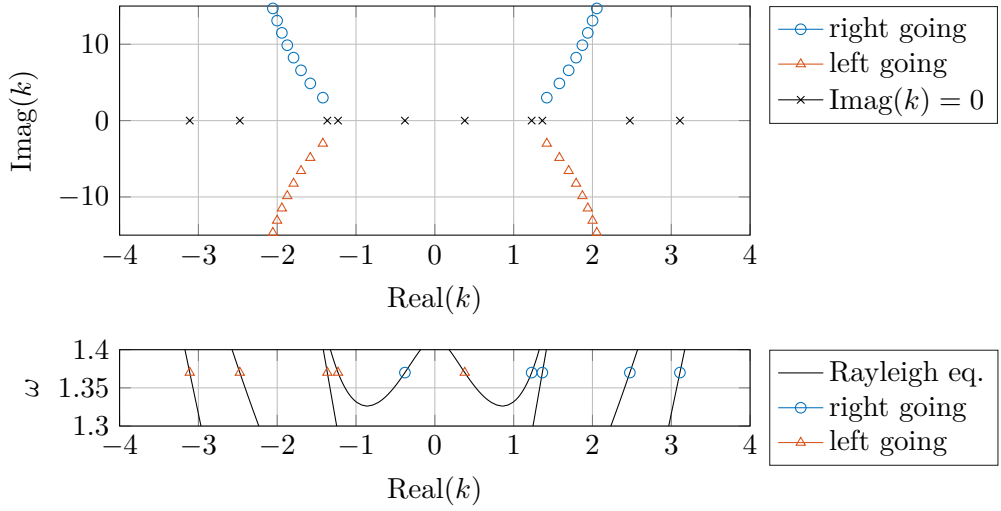}
\end{center}
\caption{\label{3_rgmode} Representation of right and left-going wavenumbers at the frequency $\omega=1.37$ with $\mu=0.25$ and $\lambda=0.31$. Right-going modes are represented by circles 
and left-going modes by triangles. Top: classification of the wavenumbers with $\Imag(k)\neq 0$. Bottom : classification of the wavenumbers with $\Imag(k)=0$ by looking at $\partial_k \omega$. }
\end{figure}

As mentioned previously, if $k_n$ is a right-going mode, $-k_n$ is also solution of the Rayleigh-Lamb equation and is then a left-going mode. We also notice, using \eqref{3_lambs}, that if $(\Xb_n,\Yb_n)$ is (up to a multiplicative constant)
the right-going Lamb mode associated to $k_n$ and $(\wt \Xb_n, \widetilde{\Yb_n})$ is the 
left-going Lamb mode associated to $-k_n$, then 
\begin{equation}
u_n=-\wt u_n, \quad v_n=\wt v_n, \quad s_n=\wt s_n, \quad t_n=-\wt t_n. 
\end{equation}
It follows that for any $(\Xb,\Yb)\in H$, there exist $(A_n)_{n\in \N^*}, (B_n)_{n\in \N^*}$ 
such that 
\begin{equation}
(\Xb,\Yb)=\sum_{n>0} A_n(\Xb_n,\Yb_n)+\sum_{n>0} B_n (\wt \Xb_n, \widetilde{\Yb_n}) = \sum_{n>0} A_n(\Xb_n,\Yb_n)+\sum_{n>0}B_n(-\Xb_n,\Yb_n). 
\end{equation}
Defining $a_n=A_n-B_n$ and $b_n=A_n+B_n$, 
\begin{equation}\label{3_decab}
\Xb=\sum_{n>0} a_n \Xb_n, \qquad \Yb= \sum_{n>0} b_n\Yb_n. 
\end{equation}
This decomposition with right-going modes is easier to handle than the full decomposition:
We prove below that we can find an explicit expression of $a_n$ and $b_n$ given $(\Xb,\Yb)$. 
We denote $\langle \cdot, \cdot\rangle$ the product defined by 
\begin{equation}
\forall \xi_i\in \text{L}^2((-h,h),\C) \qquad \langle (\xi_1,\xi_2), (\xi_3,\xi_4)\rangle =\int_{-h}^h \xi_1(z)\xi_3(z)+\xi_2(z)\xi_4(z)\dd z.
\end{equation}
Note that this product is not a scalar product since the $\xi_i$'s are
complex-valued. 
The following proposition states that families $(\Xb_n)_{n\in \N^\star}$ 
and $(\Yb_n)_{n\in \N^\star}$ are bi-orthogonal: 

\begin{prop}\label{3_vanishJn}
For every $n,m>0$, $\langle \Xb_m,\Yb_n\rangle=\delta_{n=m} J_n$ where $J_n=i\omega^2 k \,\Gamma_S$ if $n$ is a symmetric mode and $J_n=i\omega^2 k \,\Gamma_A$ if $n$ is an anti symmetric mode, with $\Gamma_S$ and $\Gamma_A$ defined in \eqref{3_crits} and \eqref{3_crita}. Especially, $J_n\neq 0$ if and only if $\omega\notin \omega_\crit$. 
\end{prop}

\begin{proof}
The proof that $\langle \Xb_m,\Yb_n\rangle=0$ if $m\neq n$ can be found in \cite{fraser1,pagneux2}. Then, using expressions \eqref{3_lambs} and \eqref{3_lamba}, we can compute $\langle \Xb_n,\Yb_n\rangle$ as in \cite{pagneux2}. 
\end{proof}

This provides a new characterization of critical points: 

\begin{coro}\label{3_caraccritic}
The three following definitions of the set of critical frequencies are equivalent:
\begin{equation}
\begin{array}{rcl} \omega_\crit&=&\{\omega\in \R_+ \,|\, \exists n\in \N^\star\quad  \langle \Xb_n,\Yb_n\rangle = 0\}\\ &=&\{\omega\in \R_+ \,|\, (\Xb_n,\Yb_n)_{n>0} \text{ does not form a complete set of functions in } H\}\\ &=& \{\omega\in \R_+\,|\, \exists n\in \N^\star \quad  \partial_{k_n}\omega=0\}.\end{array}
\end{equation}
\end{coro}

Proposition \ref{3_vanishJn} allows us to compute the coefficients in a
decomposition~\eqref{3_decab}, as 
\begin{equation}
a_n=\frac{\langle \Xb,\Yb_n\rangle}{J_n}, \qquad b_n=\frac{\langle \Yb,\Xb_n\rangle}{J_n}. 
\end{equation}

We use the modal decomposition to provide an outgoing condition for elastic waveguides.
For acoustic waves, one may ask that each modal component 
should satisfy a one dimensional Sommerfeld radiation condition (see \cite{bonnetier1} for instance). 
In the same spirit, we consider the following

\begin{defi}
A wavefield $\ub\in \text{H}^2_\loc(\Omega)$ is said to be outgoing if $\Xb$ and $\Yb$ defined in \eqref{3_XY} satisfy
\begin{equation}
\left| \langle \Yb, \Xb_n\rangle'(x)\frac{x}{|x|}-ik_n \langle \Yb,\Xb_n\rangle(x)\right| \, ,\, \left| \langle \Xb, \Yb_n\rangle'(x)\frac{x}{|x|}-ik_n \langle \Xb,\Yb_n\rangle(x)\right|\underset{|x|\rightarrow +\infty}{\longrightarrow} 0 \quad \forall n\in \N^*.
\end{equation}
\end{defi}

This condition guarantees existence and uniqueness for the source problem \eqref{3_lamb1}
as we prove below.
For every $r>0$, we set $\Omega_r:=(-r,r)\times (-h,h)$.
We consider that any source defined on $\Omega_r$ (resp. $(-r,r)$) is extended 
by $0$ on $\Omega$ (resp. $\R$).

%%--------------------------------------------------------
\begin{theorem}\label{3_solution2D}
Let $r>0$. For every $\omega\notin \omega_\crit$, $\fb=(f_1,f_2)\in \text{H}^1(\Omega_r)$ and $\bb^\top=(b^\top_1,b^\top_2), \bb^\bot=(b^\bot_1,b^\bot_2)\in \widetilde{\text{H}}^{3/2}(-r,r)$, 
the system
\begin{equation}\label{3_eq2D}
\left\{\begin{array}{cl} \nabla\cdot\sib(\ub)+\omega^2\ub=-\bm{f} & \text{ in } \Omega, \\ \sib(\ub)\cdot \nu= \bb^{\top/\bot} & \text{ on } \Omega_{\top/\bot}, \\
\ub \text{ is outgoing,} &\end{array}\right. 
\end{equation}
has a unique solution $\ub\in \text{H}^3_\loc(\Omega)$. This solution admits a Lamb-mode decomposition
\begin{equation}\label{3_sol2D}
u(x,z)=\sum_{n>0}a_n(x)u_n(z), \qquad v(x,z)=\sum_{n>0} b_n(x)v_n(z), 
\end{equation}
where $a_n, b_n$ are solutions to the decoupled Helmholtz system
\begin{equation}
\left\{\begin{array}{c} a_n''+k_n^2a_n=ik_nF_1^n-{F_2^n}', \\ b_n''+k_n^2b_n={F_1^n}'-ik_nF_2,\end{array}\right. 
\end{equation}
\begin{equation}\label{3_F1}
\text{where} \qquad F_i^n(x)=\frac{1}{J_n}\left(\int_{-h}^hf_i(x,z) u_n(z)\dd z 
+b_i^\top(x) u_n(h)+b_i^\bot(x) u_n(-h)\right), \quad i\in \{1,2\}.\,\,\,
\end{equation}
Equivalently, $a_n=G_1^n\ast F_1^n-G_2^n\ast F_2^n$ and $b_n=G_2^n\ast F_1^n-G_1^n\ast F_2^n$ with 
\begin{equation}\label{3_gi}
G_1^n(x)=\frac{1}{2}e^{ik_n|x|},\qquad  G_2^n(x)=\frac{x}{2|x|}e^{ik_n|x|}. 
\end{equation}
Moreover, there exists a constant $C>0$, which only depends on $h$, $\omega$ and $r$, 
such that 
\begin{equation}
\Vert \ub \Vert_{\text{H}^3(\Omega_r)}\leq C \left(\Vert \fb \Vert_{\text{H}^1(\Omega)}+\Vert \bb^\top \Vert_{\text{H}^{3/2}(\R)}+\Vert \bb^\bot \Vert_{\text{H}^{3/2}(\R)}\right). 
\end{equation}
\end{theorem}
%%--------------------------------------------------------

\begin{proof}
This proof is an adaptation of the proof presented in Appendix A of \cite{bonnetier1}. 

{\it Step 1.} We first show uniqueness of the solution. 
Assume that $u$ solves~\eqref{3_eq2D} with $\fb = 0$, $\bb^\top = \bb^\bot = 0$. 
The associated fields $\Xb,\Yb$ defined in~\eqref{3_XY} can be decomposed as
\begin{equation}
\Xb(x,z)=\sum_{n>0}a_n(x)\Xb_n(z), \qquad \Yb(x,z)=\sum_{n>0} b_n(x)\Yb_n(z).
\end{equation}
As proved in \cite{pagneux2}, the operators $F$ and $G$ defined in \eqref{3_eqXY} are self adjoint on $H_0$. By projecting \eqref{3_XY} with the product $\langle \cdot ,\cdot\rangle$ on $\Xb_n$ and $\Yb_n$, we see that 
\begin{equation}\label{3_proofeq1}
a_n'=ik_nb_n, \qquad b_n'=ik_na_n.
\end{equation}
Solving this system of ODE's and using the outgoing condition 
shows that $a_n=b_n=0$, leading to $\ub=0$.

{\it Step 2.}  Assume that the functions $(u,v)$ defined in \eqref{3_sol2D} is well defined.
A quick computation shows that the associated fields $(\Xb,\Yb)$ satisfy \eqref{3_eqXY} 
for every mode since 
\begin{equation}\label{3_proofeq2}
{G^n_1}'=ik_n G^n_2, \qquad {G^n_2}'=ik_nG^n_1+\delta_0. 
\end{equation}
Assuming that $|x|>r$, we also see that 
\begin{multline}
 \left|\langle\bm{X},\bm{Y}_n\rangle'(x)\frac{x}{|x|}-ik_n\langle\bm{X},\bm{Y}_n\rangle(x)\right| \\=\left|\left(\frac{x}{|x|}{G_1^n}'-ik_nG_1^n\right)\ast F_1^n(x)-\left(\frac{x}{|x|}{G_2^n}'-ik_nG_2^n\right) \ast F_2^n(x)\right|=0.
\end{multline}
Repeating this computation for $\langle \Yb,\Xb_n\rangle$, 
shows that $(u,v)$ satisfies the outgoing condition. 

{\it Step 3.}
We  prove that the functions given by~\eqref{3_sol2D} are well-defined, in other words
that the series in \eqref{3_sol2D} converge. 
We know from~\cite{royer1} that the number of evanescent and propagative modes is finite,
so we only need to study  the convergence of the inhomogeneous modes.
We begin by noticing that if $k_n$ is an inhomogeneous mode, then $k_m=-\overline{k_n}$ 
also satisfies the dispersion relation, and $X_m=\overline{X_n}$ and $Y_m=\overline{Y_n}$. 
We index the subset of inhomogeneous wavenumbers with positive real part by $j\in \N^*$, 
and all the asymptotic comparison are now meant when $j\to +\infty$. 
Let $N=2j-1/2$ (resp. $N=2j+1/2$) if $k_j$ is associated to a symmetric 
(resp. antisymmetric) mode. 
Using \cite{merkulov1}, we know that 
\begin{equation}
hk_j=\frac{1}{2}\ln(2\pi N)+i\frac{\pi N}{2}-i\frac{\ln(2\pi N)}{2\pi N}+O\left(\frac{1}{N}\right).
\end{equation}
Let $p_j$ and $q_j$ be the quantities defined in \eqref{3_disps}. We notice that $p_j,q_j\sim -ik_j$. Since $(p_j-ik_j)(p_j+ik_j)=\omega^2/(\lambda+2\mu)$ and $(q_j-ik_j)(q_j-ik_j)=\omega^2/\mu$ it follows that 
\begin{equation}
p_j=-ik_j+\underbrace{\frac{\omega^2}{(\lambda+2\mu)\pi}}_{c_p}\frac{1}{N}+o\left(\frac{1}{N}\right), \quad q_j=-ik_j+\underbrace{\frac{\omega^2}{\mu\pi}}_{c_q}\frac{1}{N}+o\left(\frac{1}{N}\right).
\end{equation}
We notice that 
\begin{equation}
\sin(ik_j z) =\frac{iz}{2|z|}\exp\left(\frac{|z|}{h}\left(\frac{1}{2}\ln(2\pi N)+i\frac{\pi N}{2}-i\frac{\ln(2\pi N)}{2\pi N}\right)\right)+O\left(\frac{1}{N^{1-|z|/2h}}\right),\end{equation}
\begin{equation}
\cos(ik_j z) =\frac{1}{2}\exp\left(\frac{|z|}{h}\left(\frac{1}{2}\ln(2\pi N)+i\frac{\pi N}{2}-i\frac{\ln(2\pi N)}{2\pi N}\right)\right)+O\left(\frac{1}{N^{1-|z|/2h}}\right),\end{equation}
and if $\alpha$ stands for $p$ or $q$, 
\begin{equation}
 \cos(\alpha_j z) \sim \cos(ik_jz)+\frac{c_\alpha}{N}z\sin(ik_j z), \quad 
 \sin(\alpha_j z)\sim-\sin(ik_jz)+\frac{c_\alpha}{N}z\cos(ik_j z). 
\end{equation}
Using the definition of the symmetric modes \eqref{3_lambs}, we find that $(\Xb_j(z); \Yb_j(z))\sim $
\begin{equation*}
\left(\begin{array}{c} h-|z|\\-i\mu \pi N(z-hz/|z|)\\ \mu\pi N(h-|z|)\\ i(z-hz/|z|) \end{array}\right)\frac{\pi^3 N^2(c_p-c_q)}{16}\exp\left(\left(\frac{|z|}{h}+1\right)\left(\frac{\ln(2\pi N)}{2}+\frac{i\pi N}{2}-\frac{i\ln(2\pi N)}{2\pi N}\right)\right).
\end{equation*}
%\begin{eqnarray*}
%u_j(z) &\sim& \frac{\pi^3 N^2(c_p-c_q)}{16}(h-|z|)\exp\left(\left(\frac{|z|}{h}+1\right)\left(\frac{\ln(2\pi N)}{2}+\frac{i\pi N}{2}-\frac{i\ln(2\pi N)}{2\pi N}\right)\right),
%\\
%v_j(z) &\sim& \displaystyle \frac{i\pi^3 N^2(c_p-c_q)}{16}\left(z-h\frac{z}{|y|}\right)\exp\left(\left(\frac{|z|}{h}+1\right)\left(\frac{\ln(2\pi N)}{2}+\frac{i\pi N}{2}-\frac{i\ln(2\pi N)}{2\pi N}\right)\right),
%\\
%t_j(z) &\sim& \displaystyle \frac{-i\mu\pi^4 N^3(c_p-c_q)}{16}\left(z-h\frac{z}{|z|}\right)\exp\left(\left(\frac{|z|}{h}+1\right)\left(\frac{\ln(2\pi N)}{2}+\frac{i\pi N}{2}-\frac{i\ln(2\pi N)}{2\pi N}\right)\right),
%\\
%s_j(z) &\sim& \displaystyle \frac{-\mu \pi^4 N^3(c_p-c_q)}{16}\left(h-|z|\right)\exp\left(\left(\frac{|z|}{h}+1\right)\left(\frac{\ln(2\pi N)}{2}
%+\frac{i\pi N}{2}-\frac{i\ln(2\pi N)}{2\pi N}\right)\right),
%\end{eqnarray*}
and it follows that
\begin{equation*}
\Vert u_j\Vert_{\text{L}^2(-h,h)}, \Vert v_j\Vert_{\text{L}^2(-h,h)}
\sim
\frac{\pi^4 h^{3/2}N^3|c_p-c_q|}{4\ln(2\pi N)^{3/2}}, 
\quad \Vert u_j\Vert_{\text{L}^\infty(-h,h)}, \Vert v_j\Vert_{\text{L}^\infty(-h,h)}
\sim  
\frac{\pi^5 h N^2|c_p-c_q|}{8}, 
\end{equation*}
\begin{equation*}
\Vert t_j\Vert_{\text{L}^2(-h,h)}, \Vert s_j\Vert_{\text{L}^2(-h,h)}
\sim
\frac{\pi^5 h^{3/2}\mu N^4|c_p-c_q|}{4\ln(2\pi N)^{3/2}},
\quad 
|J_j|
\sim \frac{h^2\pi^6 N^5\omega^2|c_p-c_q|}{8}.
\end{equation*}
%\begin{eqnarray}
%\label{3_inhomonuv}
%\Vert u_j\Vert_{\text{L}^2(-h,h)}, \Vert v_j\Vert_{\text{L}^2(-h,h)}
%&\sim& 
%\frac{\pi^4 h^{3/2}N^3|c_p-c_q|}{4\ln(2\pi N)^{3/2}}, 
%\\ 
%\label{3_inhomonst}
%\Vert t_j\Vert_{\text{L}^2(-h,h)}, \Vert s_j\Vert_{\text{L}^2(-h,h)}
%&\sim& 
%\frac{\pi^5 h^{3/2}\mu N^4|c_p-c_q|}{4\ln(2\pi N)^{3/2}},
%\\
%\label{3_inhomoninf}
%\Vert u_j\Vert_{\text{L}^\infty(-h,h)}, \Vert v_j\Vert_{\text{L}^\infty(-h,h)}
%&\sim& 
%\frac{\pi^5 h N^2|c_p-c_q|}{8}, 
%\\
%\label{3_inhomojn}
%|J_j|
%&\sim& \frac{h^2\pi^6 N^5\omega^2|c_p-c_q|}{8}.
%\end{eqnarray}
Similar estimates can be derived for the antisymmetric modes, which yield the same asymptotic 
behaviors. 
Defining 
\begin{equation}
a_j=G_1^j\ast F_1^j-G_2^j\ast F_2^j, \qquad b_j=G_2^j\ast F_1^j-G_1^j\ast F_2^j,
\end{equation}
we see using Young's inequality that 
\begin{equation}
\Vert a_j \Vert_{\text{L}^2(-r,r)}\leq \Vert G_1^j \Vert_{\text{L}^1(-r,r)}\Vert F_1^j \Vert_{\text{L}^2(-r,r)}+ \Vert G_2^j \Vert_{\text{L}^1(-r,r)}\Vert F_2^j \Vert_{\text{L}^2(-r,r)}.
\end{equation}
From the asymptotics of $k_j$ it follows that
$\Vert G_1^j \Vert_{\text{L}^1(\R)},\Vert G_2^j \Vert_{\text{L}^1(\R)}\leq 1/N$.
Thus if $u^{p}_j$ denotes a primitive of $u_j$, we see that
\begin{equation}
\int_{-h}^h f_1(x,z) u_j(z)\dd y =\left[f_1(x,z)u_j^{p}(z)\right]^{z=h}_{z=-h}
-\int_{-h}^h \partial_{z} f_1(x,z) u_j^{p1}(z) \dd z.
\end{equation}
Using the previous estimates, we find that $u^{p}_j(z) \sim \tfrac{2}{i\pi N}u_j(z)$
and so there exists a constant $c_1>0$, that depends on $\omega$ and $h$, such that 
\begin{equation}\label{3_contrF1}
\Vert F_1^j\Vert_{\text{L}^2(-r,r)}\leq \frac{c_1}{N^3} \left(\Vert f_1 \Vert_{\text{H}^1(\Omega)}+\Vert b_1^\top\Vert_{\text{L}^2(\R)}+\Vert b_1^\bot\Vert_{\text{L}^2(\R)}\right). 
\end{equation}
We obtain a similar estimate for $F_2$. Finally, using the triangular inequality, 
\begin{equation}\label{3_IT}
\Vert u \Vert_{\text{L}^2(\Omega_r)}\leq \sum_{n | k_n \in \R, i\R} \Vert a_n\Vert_{\text{L}^2(-r,r)}\Vert u_n \Vert_{\text{L}^2(-h,h)}+ 2\sum_{j\in \N^*} \Vert a_j\Vert_{\text{L}^2(-r,r)}\Vert u_j \Vert_{\text{L}^2(-h,h)},
\end{equation}
which leads to
\begin{multline}\label{3_estim_u}
\Vert u \Vert_{\text{L}^2(\Omega)}\leq \Bigg[\sum_{n | k_n\in \R, i\R} \frac{2r}{|J_n|}\left(\Vert u_n \Vert_{\text{L}^2(-h,h)}+\Vert v_n \Vert_{\text{L}^2(-h,h)}\right)\Vert u_n \Vert_{\text{L}^2(-h,h)}\\+2c_1\sum_{j\in \N^*} \frac{1}{(2j\pm 1/2)\ln(2j\pm 1/2)^{3/2}}\Bigg]\left(\Vert \fb \Vert_{\text{H}^1(\R)}+\Vert \bb^\top \Vert_{\text{L}^2(\R)}+\Vert \bb^\bot \Vert_{\text{L}^2(\R)}\right),
\end{multline}
A similar control holds for $v$. Elliptic regularity results (see e.g.~\cite{grisvard1}) 
show that there exists a constant $c_3$ depending on $r$ such that 
\begin{equation}
\Vert \ub \Vert_{\text{H}^3(\Omega_r)}\leq c_3\left(\Vert \ub \Vert_{\text{L}^2(\Omega_r)}+\Vert \fb \Vert_{\text{H}^1(\R)}+\Vert \bb^\top \Vert_{\text{H}^{3/2}(\R)}+\Vert \bb^\bot \Vert_{\text{H}^{3/2}(\R)}\right),
\end{equation}
which together with~\eqref{3_estim_u} conclude the proof.
\end{proof}
\begin{rem}
This result is probably not optimal: indeed, in the scalar case one can merely
assume that the source term lies in $\text{L}^2(\Omega_r)$ and obtain a solution
in $\text{H}^2_\loc(\Omega)$ (see \cite{bonnetier1}).
However, in the present case, the Lamb modes are not orthogonal and Parseval equality
does not hold, so that in the above proof, we controlled terms using the triangular inequality,
which may lead to a loss of accuracy. 
We can see in the proof that the extra regularity of the source terms is needed to 
derive \eqref{3_contrF1}, which in turn yields the convergence of the series \eqref{3_IT}. Providing adaptation to elastic waveguides, the theory developed in~\cite{nazarov1} may be better adapted to treat source terms with lower regularity.
\end{rem}

To conclude, in this section we have constructed  an explicit solution of the elasticity problem 
in a regular waveguide, and have shown that its norm is controlled by that of the source terms. 
Such estimates will be useful in the following to perform the Born approximation. 

%%-------------------------------------------------------------------------------------------

\section{Forward source problem in a regular 3D plate}
%%-------------------------------------------------------------------------------------------

In this section, we are interested in the forward source problem in a three-dimensional regular waveguide with two infinite dimensions. Our motivation comes from the experiments reported in~\cite{balogun1}, where the authors try to reconstruct width defects in thin elastic plates.
The propagation of waves in three-dimension waveguides with one infinite dimension such as pipes 
or air ducts is a direct generalization of the two dimensional case presented in the previous 
section, see for instance~\cite{baronian2}. 
Two main issues are at stake. First, as mentioned in \cite{legrand1}, 
in addition to longitudinal and transverse modes, one needs to take into account 
horizontal shear modes in order to form a complete modal basis. 
Second, one would like to generalize the $(\Xb,\Yb)$ formulation
of Definition~\ref{3_lambmode} to 3D.
\medskip

Given an elastic wavefield ${\bf u} = (u,v,w)$ that propagates in a three dimension plate,
our main contribution consists in introducing two auxiliary variables $\alpha$ and $\beta$ in \eqref{3_ab}
that only depend on $u$ and $v$, which allow the decoupling of the equations 
of elasticity. We show that $(\alpha,w)$ can be decomposed using Lamb modes, 
while $\beta$ represents the horizontal shear modes. 
This allows us to obtain a generalization of Theorem \ref{3_solution2D} 
to three dimensional plates. 

%%---------------------------------------------------------------
\subsection{Decoupling of the linear elastic equation}
%%---------------------------------------------------------------

Let us consider a 3D infinite elastic plate $\Omega=\R^2\times (-h,h)$,
where $h>0$ is half of the waveguide thickness. 
For every $r>0$, we define $\Omega_r=B_2(0,r)\times (-h,h)$ where 
$B_2(0,r)$ is the ball in $\R^2$ centered at $(0,0)$ with radius $r$. 
A point $(x,y,z) \in \Omega$ will be denoted by $(\xb,z)$,
and the elastic displacement by $\ub=(u,v,w)$. 
Given a frequency $\omega\in \R$ and given $(\lambda,\mu)$ the Lam\'e 
coefficients of the elastic waveguide, the wavefield $\ub$ satisfies  
\begin{equation}\label{3_lamb3D}
\nabla\cdot\sib(\ub)+\omega^2\ub=-\bm{f}\qquad \text{ in } \Omega, 
\end{equation}
where $\fb= (f_1,f_2,f_3)$ is a source term and $\sib(\ub)$ is the stress tensor defined by
\begin{equation}
\sib(\ub)=\left(\begin{array}{ccc} \begin{array}{l} (\lambda+2\mu)\partial_x u \\ \quad +\lambda\partial_y v+\lambda \partial_z w\end{array} & \mu \partial_y u+\mu \partial_x v & \mu \partial_z u+\mu \partial_x w\\ \mu \partial_y u+\mu \partial_x v  & \begin{array}{l} (\lambda+2\mu) \partial_y v \\ \quad +\lambda\partial_x u +\lambda\partial_z w \end{array}& \mu \partial_z v +\mu \partial_y w \\ \mu\partial_z u +\mu \partial_x w & \mu \partial_zv+\mu\partial_y w & \begin{array}{l} (\lambda+2\mu)\partial_z w\\ \quad + \lambda\partial_x u +\lambda\partial_y v\end{array} \end{array}\right).
\end{equation}
In the following, we study the case of Neumann boundary conditions
\begin{equation}\label{3_neumann3D} \sib(\ub)\cdot \nu=\bb^\top \quad  \text{ on } \partial\Omega_\top, \qquad \sib(\ub)\cdot \nu=\bb^\bot \quad  \text{ on } \partial\Omega_\bot, \end{equation}
where $\bb^\top=(b_1^\top,b_2^\bot,b_3^\bot)$ and $\bb^\bot=(b_1^\bot,b_2^\bot,b_3^\bot)$ 
are boundary source terms. 
However, our analysis applies also to the case of Dirichlet or Robin boundary conditions. 
We represent the set-up in Figure \ref{3_guide3D}.

\begin{figure}[h]
\begin{center}
\begin{tikzpicture}[scale=0.7]
%plaques
\draw  (0,0)-- (11,0);
\draw (0,2)-- (11,2);
\draw (0,2)-- (3.0244955610383184,3.6028808633601677);
\draw [dashed] (0,0)-- (3.0244955610383184,1.5877554061936634);
\draw  (0,0)-- (0,2);
\draw [ultra thick] [->] (0,2) -- (0,2.8);
\draw (0,2.8) node[above]{$e_z$};
\draw [ultra thick] [->] (0,2) -- (0.74,2.39);
\draw (0.74,2.39) node[above]{$e_y$};
\draw [ultra thick] [->] (0,2) -- (0.8,2);
\draw (0.8,1.95) node[below]{$e_x$};
\draw (9,1) node{$\Omega$};
\draw (0,0) node[left]{$-h$}; 
\draw (0,2) node[left]{$h$}; 
%rectangles 
\fill[fill opacity=0.2] (1.21,2.64)--(2.57,3.35)--(4.69,3.35)--(3.33,2.64) -- cycle;
\draw (2.9,3.05) node{$\bb^\top$}; 
\fill[fill opacity=0.2] (1.81,0)--(4,0)--(5.49,0.78)--(3.3,0.78) -- cycle;
\draw (3.5,0.4) node{$\bb^\bot$}; 
%patate
\draw [opacity=0.5,rotate around={-28.70196203083524:(5.997085152361325,1.3310013712642719)}] (5.997085152361325,1.3310013712642719) ellipse (0.7562594823025554cm and 0.5671521672338097cm);
\draw [rotate around={-28.70196203083524:(5.997085152361325,1.3310013712642719)}] (5.997085152361325,1.3310013712642719) node{$\fb$};
\fill[fill opacity=0.2]  plot[shift={(6.35,1.65)},domain=-1.8274502127484897:2.192702065520201,variable=\t]({0.7893522173763262*1.3874997236144535*cos(\t r)+-0.6139406135149204*0.7906677450296004*sin(\t r)},{0.6139406135149204*1.3874997236144535*cos(\t r)+0.7893522173763262*0.7906677450296004*sin(\t r)})-- plot[shift={(5.997085152361325,1.3310013712642719)},domain=3.20683662222621:5.964120372135418,variable=\t]({0.8771297184667898*0.7562594823025554*cos(\t r)+0.4802535340654666*0.5671521672338097*sin(\t r)},{-0.4802535340654666*0.7562594823025554*cos(\t r)+0.8771297184667898*0.5671521672338097*sin(\t r)})--cycle;
\end{tikzpicture}
\end{center}
\caption{\label{3_guide3D} The three dimensional plate $\Omega$. Elastic wavefields are generated 
by an internal source term $\fb$ and by boundary source terms $\bb^\top$ and $\bb^\bot$.  }
\end{figure}
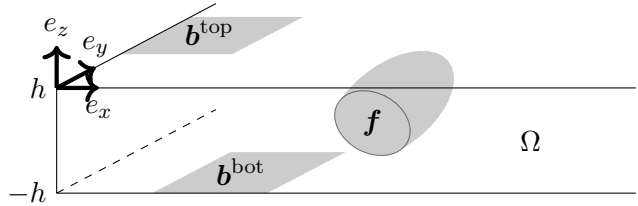

To adapt the $\Xb/\Yb$ formulation to 3D, we introduce the following notations~: for a vector field
$\bm{g}=(g_1,g_2,g_3)$ and a scalar field $g$ we set
\begin{equation}
\di_2(\bm{g})=\partial_x g_1+\partial_y g_2, \quad \curl_2(\bm{g})=\partial_x g_2-\partial_y g_1, \quad \Delta_2(g)= \partial_{xx}g +\partial_{yy}g.
\end{equation}
Given ${\bf u} = (u,v,w) \in H^1_{\loc}(\Omega)$ we define
\begin{equation}\label{3_ab}
\alpha=\di_2(\ub), \qquad \beta=\curl_2(\ub). 
\end{equation}
Note that $\alpha$ and $\beta$ only involve the in-plane components of $\ub$.
The following proposition shows how these new variables decouple 
the elasticity system: 
%%---------------------------------------------------------------
\begin{prop}
If $\ub$ is solution of \eqref{3_lamb3D} with boundary conditions \eqref{3_neumann3D}, 
then $(\alpha,w)$ satisfies 
\begin{equation} \label{3_3Ddec1} \left\{\begin{array}{cl} 
(\lambda+2\mu)\Delta_2 \alpha+\mu \partial_{zz} \alpha +\omega^2 \alpha +(\lambda+\mu)\Delta_2 \partial_z w =-\text{div}_2(\fb) & \text{ in } \Omega,  \\ \displaystyle (\lambda+\mu)\partial_z \alpha +(\lambda+2\mu) \partial_{zz} w +\mu\Delta_2 w +\omega^2 w = -f_3 & \text{ in } \Omega, \\ \displaystyle \partial_z \alpha +\mu\Delta_2 w = \di_2 (\bb^{\top/\bot})& \text{ on } \partial\Omega_{\top/\bot}, \\ \displaystyle (\lambda+2\mu)\partial_z w+\lambda \alpha= b_3^{\top/\bot} & \text{ on } \partial\Omega_{\top/\bot}, \end{array}\right. 
\end{equation}
while $\beta$ satisfies 
\begin{equation} \label{3_3Ddec2}
\left\{ \begin{array}{cl}  \displaystyle\mu\Delta_2 \beta+\mu \partial_{zz} \beta +\omega^2 \beta =-\text{curl}_2(\fb) & \text{ in } \Omega, \\  \displaystyle \mu\partial_z \beta =\curl_2(\bb^{\top/\bot}) & \text{ on } \partial\Omega_{\top/\bot}.\end{array}\right. 
\end{equation}
\end{prop}
%%---------------------------------------------------------------
\begin{proof}
If we denote $L_i$ the lines of \eqref{3_lamb3D} and $B_i$ the lines of \eqref{3_neumann3D}, we compute $\partial_x L_1+\partial_y L_2$, $\partial_x B_1+\partial_y B_2$ and with $L_3$ and $B_3$, we find \eqref{3_3Ddec1}. Then, $\partial_x L_2-\partial_y L_1$ and $\partial_x B_2-\partial_y B_1$ give \eqref{3_3Ddec2}. 
\end{proof}
\medskip

We start with the study of equation \eqref{3_3Ddec2}, which is a Helmholtz
equation, similar to that of acoustic waveguides, for the function $\beta$. 
Inspired by \cite{bonnetier1}, we introduce a decomposition
of $\beta$  as a sum of horizontal shear modes. 

\begin{defi}
For every $n\in \N$, we define $\kappa_n^2=\omega^2/\mu-n^2\pi^2/4h^2$ with $\text{Real}(\kappa_n)\geq 0$ and $\text{Imag}(\kappa_n)\geq 0$. 
We define the $n$-th shear horizontal mode (SH mode) $\varphi_n$ by 
\begin{equation}
\varphi_n(z):= \left\{\begin{array}{cl}1/\sqrt{2h} & \text{ if } n=0, \\ \frac{1}{\sqrt{h}}\cos\left(\frac{n\pi(z+h)}{2h}\right) & \text{ else.}\end{array}\right. 
\end{equation}
The sequence $(\varphi_n)_{n \geq 0}$ defines an orthonormal basis of 
$\text{L}^2(-h,h)$ for the scalar product
\begin{equation}
(g_1\,|\,g_2):=\int_{-h}^h g_1(x)g_2(x)\dd x .
\end{equation} 
\end{defi}

Classical results on waveguides (see e.g.~\cite{bourgeois1})
show that if one imposes a Sommerfeld radiation condition~\cite{sommerfeld1}, 
the problem \eqref{3_3Ddec2} is well-posed except 
for the  frequencies 
\begin{equation}
\omega_n^{sh} = \frac{\sqrt{\mu}\pi}{2h}n, \quad n \geq 0. 
\end{equation}
More precisely, when $\omega \notin \omega_\text{crit}^{sh} := \{ \omega_n^{sh}, n \geq 0\}$,
the following result holds:

\begin{prop}\label{3_prop3D2}
For every $\omega\notin \omega_{\text{crit}}^{sh}$, $\fb\in \text{H}^1(\Omega_r)$ and $\bb^\top,\bb^\bot\in \widetilde{H}^{3/2}(-r,r)$, the problem 
\begin{equation}\label{3_eq_Hbeta}
\left\{ \begin{array}{cl} \mu\Delta_2 \beta+\mu \partial_{zz} \beta +\omega^2 \beta =-\curl_2(\fb) & \text{ in } \Omega, \\ \mu\partial_z \beta =\curl_2(\bb^{\bot/\top}) & \text{ on } \partial\Omega_{\bot/\top}, \\ \sqrt{R}\left[\partial_r-i\kappa_n\right]\left(\beta \, | \, \varphi_n\right)(Re^{i\theta}) \underset{R \rightarrow +\infty}{\longrightarrow} 0 \quad \forall n\geq 0 \quad \forall \theta \in (0,2\pi),   \end{array}\right. 
\end{equation}
has a unique solution $\beta\in \text{H}_\loc^2(\Omega)$ which decomposes as  
\begin{eqnarray} \label{3_decomp_beta}
\beta(\xb,z) &=& \sum_{n\geq 0} -(\Gamma^n\ast F_{sh}^n)(\xb) \varphi_n(z),
\end{eqnarray} 
where $\Gamma^n$ denotes the Hankel function of the first kind $\Gamma^n(\xb) = -\tfrac{i}{4}\hau_0(\kappa_n|\xb|)$, and where
\begin{eqnarray*}
F_{sh}^n &=&
\frac{1}{\mu}\left(\int_{-h}^h \curl_2(\fb)\varphi_n+\curl_2(\bb^\top)\varphi_n(1)
+\curl_2(\bb^\bot)\varphi_n(0)\right). 
\end{eqnarray*}
\end{prop}
\begin{proof}
We follow the exact same steps as Appendix A in \cite{bonnetier1}. 
Since the $\varphi_n$ form an orthonormal basis, any function
$\beta$ can be decomposed as $\beta=\tsum \beta_n\varphi_n$. 
Projecting on the SH modes, the problem~\eqref{3_eq_Hbeta} is equivalent 
to the collection of problems indexed by $n \in \N$
\begin{equation}
\left\{\begin{array}{cl} \Delta_2 \beta_n+\kappa_n^2b_n=g_{sh}^n & \text{ in } \R^2 \\ \sqrt{r}\left[\partial_r -i\kappa_n\right]\beta_n(Re^{i\theta}) \underset{r\rightarrow +\infty}{\longrightarrow} 0 \quad \forall \theta\in (0,2\pi). \end{array}\right. 
\end{equation}
As $\Delta \Gamma^n+\kappa_n^2\Gamma^n=\delta_0$~\cite{williams1},
each function $\beta_n$ can be expressed as the convolution $\Gamma^n * F^n_{sh}$.
The series~\eqref{3_decomp_beta} can be shown to converge in $\text{L}^2_\loc(\Omega)$
and provides a solution in the sense of distributions to~\eqref{3_eq_Hbeta}.
Elliptic regularity allows then to show that $\beta$ is actually in $\text{H}^2_{\loc}(\Omega)$.
\end{proof}
\medskip

%%--------------------------------------------------------

Next, we study equation \eqref{3_3Ddec1}, which resembles the  
two dimensional elasticity system studied in section \ref{3_section2D}. 
We propose to adapt the $\Xb/\Yb$ formulation to this new equation. We define the variables 

\begin{equation}\label{3_XY3D}
t=\mu \partial_z \alpha+\mu \Delta_2 w, \quad s= (\lambda+2\mu) \alpha +\lambda\partial_z w, \quad \Xb=(\alpha,t), \quad \Yb=(-s,w). 
\end{equation}

Then, we adapt Proposition \ref{3_equivalenceXY}: 
\begin{prop} The system \eqref{3_3Ddec1} is equivalent to 
\begin{equation}\label{3_eq3D_opform} \left(\begin{array}{c} \bm{X} \\ \Delta_2\Yb \end{array}\right)=
{\mathcal L}(\Xb,\Yb)
+\left(\begin{array}{c} 0 \\ -f_3-\delta_{z=h}b_3^\top -\delta_{z=-h}b_3^\bot\\ \text{div}_2(\fb) +\delta_{z=h}\di_2(\bb^\top)+\delta_{z=-h}\di_2(\bb^\bot)\\ 0 \end{array} \right) \quad \text{ in } \Omega,
\end{equation}
with $B1(\bm{X})=B2(\bm{Y})=0$, where ${\mathcal L}(\Xb,\Yb) = (F(\bm{Y}) ; G(\bm{X}))$,
and $F$, $G$, $B_1$ and $B_2$ are the same matrix operators as those defined in Proposition \ref{3_equivalenceXY} in \eqref{3_FG} and \eqref{3_B1B2}. 
\end{prop}
We thus may use Lamb modes to diagonalize the operator ${\mathcal L}$ as in the 2D
situation, and obtain in this way a result similar to Theorem~\ref{3_solution2D}.
Let $u_n, v_n, k_n$ be defined as in section~\ref{3_section2D} and assume that 
$\omega \notin \omega_{\text{crit}}$.

%%----------------------------------------------------
\begin{prop}\label{3_prop3D1}
For every $\omega\notin \omega_{\text{crit}}$, $\fb\in \text{H}^1(\Omega_r)$ and $\bb^\top, \bb^\bot\in \widetilde{H}^{3/2}(-r,r)$, the problem 
\begin{equation}\label{3_eq3D1rd}
\left\{\begin{array}{cl} 
(\lambda+2\mu)\Delta_2 \alpha+\mu \partial_{zz} \alpha +\omega^2 \alpha +(\lambda+\mu)\Delta_2 \partial_z w =-\text{div}_2(\fb) & \text{ in } \Omega,  \\ \displaystyle (\lambda+\mu)\partial_z \alpha +(\lambda+2\mu) \partial_{zz} w +\mu\Delta_2 w +\omega^2 w = -f_3 & \text{ in } \Omega, \\ \displaystyle \partial_z \alpha +\mu\Delta_2 w = \di_2 (\bb^{\top/\bot})& \text{ on } \partial\Omega_{\top/\bot}, \\ \displaystyle (\lambda+2\mu)\partial_z w+\lambda \alpha= b_3^{\top/\bot} & \text{ on } \partial\Omega_{\top/\bot}, \\
\sqrt{R}\left[\partial_r-ik_n\right]\langle\Yb,\Xb_n\rangle(Re^{i\theta})\underset{R\rightarrow +\infty}{\longrightarrow} 0 \quad \forall n>0 \quad \forall \theta\in (0,2\pi), & \\ \sqrt{R}\left[\partial_r-ik_n\right]\langle\Xb,\Yb_n\rangle(Re^{i\theta})\underset{R\rightarrow +\infty}{\longrightarrow} 0 \quad \forall n>0 \quad \forall \theta\in (0,2\pi), \end{array}\right. 
\end{equation}
has a unique solution $(\alpha,w)\in \text{H}^2_\loc(\Omega)\times \text{H}^3_\loc(\Omega)$ which decomposes as 
\begin{equation}\label{3_soleq3D1}
\alpha(\xb,z)=\sum_{n>0} ((k_n^2F_3^n +ik_n F_1^n)\ast G^n)u_n(z), \quad w(\xb,z)=\sum_{n>0} ((-ik_n F_3^n +F_1^n)\ast G^n)v_n(z),
\end{equation}
where $G^n(\xb)=-\tfrac{i}{4}\hau_0(k_n|\xb|)$ and 
\begin{equation}
F_1^n=\frac{1}{J_n}\left(\int_{-h}^h \di_2(\fb)u_n +\di_2(\bb^\top)u_n(h)+\di_2(\bb^\bot)u_n(-h)\right),
\end{equation}
\begin{equation}
F_3^n=\frac{1}{J_n}\left(\int_{-h}^h f_3v_n +b_3^\top v_n(h)+b_3^\bot v_n(-h)\right). 
\end{equation}
\end{prop}
%%----------------------------------------------------

\begin{proof}
The proof is very similar to that of Theorem \ref{3_solution2D}. We start by
decomposing $(\Xb, \Yb)$ in the form
\begin{eqnarray} \label{3_series_3D}
\Xb(x,y,z) \;=\; \sum_{n > 0} a_n(x,y)\Xb_n(z),
&\quad&
\Yb(x,y,z) \;=\; \sum_{n > 0} b_n(x,y)\Yb_n(z).
\end{eqnarray}
Injecting in~\eqref{3_eq3D_opform}, one obtains, instead of \eqref{3_proofeq1} in the 2D case, 
\begin{equation}
a_n=ik_nb_n, \quad \Delta_2 b_n=ik_na_n. 
\end{equation}
If \eqref{3_soleq3D1} is well defined, it satisfies \eqref{3_eq3D1rd} since $\Delta_2 G^n=-k_n^2 G^n +\delta_0$. The study of the asymptotic behavior of $a_n$ and $b_n$ can be performed
as in section~\ref{3_section2D}, using the fact that
\begin{equation}
\Vert k_n^2G^n\Vert_{\text{L}^1(B(0,r))}=\cO(1), \quad \Vert ik_nG^n\Vert_{\text{L}^1(B(0,r))}=\cO(1/n), \quad \Vert G^n\Vert_{\text{L}^1(B(0,r))}=\cO(1/n^2).
\end{equation}
It follows that the sum of series~\eqref{3_series_3D} are in $\text{L}^2_\loc(\Omega)$ and provide a solution  of~\eqref{3_eq3D1rd} in the sense of distributions.
Rewriting this system as an elliptic system 
\[
\left\{\begin{array}{cl}
(\lambda + 2\mu) \Delta_2 \alpha + \mu \partial_{zz}(\alpha)= R(w,{\bf f})
&\quad \text{in}\; \Omega, 
\\
(\lambda + 2\mu) \partial_{zz} w + (\lambda + 2\mu)\Delta_2 w= S(\alpha,f_3)
&\quad \text{in}\; \Omega,
\\
\partial_z \alpha = r(\Delta_2 w, {\bf b}^{\top/\bot})
&\quad \text{on}\; \partial \Omega_{\top/\bot},
\\
\partial_z w = s(\alpha, b_3^{\top/\bot})
&\quad \text{on}\; \partial \Omega_{\top/\bot},
\end{array}\right.
\]
using elliptic regularity~\cite{grisvard1} and a bootstrap argument, one further infers
that $\alpha, w \in \text{H}^2_{loc}(\Omega)$, which concludes the proof of the Proposition. 
\end{proof}

%%--------------------------------------------------------
\subsection{Helmholtz-Hodge decomposition}
%%--------------------------------------------------------

Now that equations \eqref{3_3Ddec1} and \eqref{3_3Ddec2} are solved, 
we return to equations \eqref{3_lamb3D} and \eqref{3_neumann3D}. 
We need to ensure that given the expressions of $\alpha$ and $\beta$, 
we can recover a unique expression for $u$ and $v$. 
To this end, we use the Helmholtz-Hodge decomposition, which states that under certain conditions, a vector field can be decomposed in a unique way as the sum of a curl-free and divergence-free
fields.
This decomposition is mostly used in fluid mechanics to analyze three dimensional vector fields 
(see for instance \cite{bhatia1,chorin1}). In our case, we apply it to two dimensional vector
fields, since we are only interested in finding a link between $(u,v)$ and $(\alpha,\beta)$. 
We give the corresponding statement below, the proof of which and be found in~\cite{ceccon1}
concerning existence, while uniqueness is addressed in~\cite{wiebel1}. 
\begin{prop}
Every vector field $\bm{\xi}\in \text{H}^1_\loc(\R^3,\C^2)$, vanishing at infinity, can be uniquely decomposed as $\bm{\xi}=\bm{d}+\bm{c}$ where $\text{curl}_2(\bm{d})=0$ and $\text{div}_2(\bm{c})=0$. The couple $(\bm{d},\bm{c})$ is called the Helmholtz-Hodge decomposition (HHD) of $\bm{\xi}$. Moreover, $\bm{\xi}$ is uniquely determined by $\text{div}_2(\bm{\xi})$ and 
$\text{curl}_2(\bm{\xi})$.
\end{prop}

Providing enough regularity on $\alpha$ and $\beta$, one can compute $(u,v)$ using the formula~\cite{ceccon1}
\begin{equation}
(u,v)=-\nabla (\mathcal{G}(\alpha))+\nabla\times (\mathcal{G}(\beta)),
\end{equation}
where $\mathcal{G}$ represents the Newtonian potential operator which convolves each function with $x\mapsto \log(|\xb|)/(2\pi)$. However, in the following, we will not need to use this formula. Indeed, we exhibit expressions of $u$ and $v$ that satisfy $\text{div}_2(\ub)=\alpha$ and $\text{curl}_2(\ub)=\beta$ and thus are the ones we look for thanks to the previous uniqueness result. 

We introduce an outgoing radiation condition for 3D wavefields: 

\begin{defi}
A wavefield $\ub\in \text{H}^2_{\loc}(\Omega)$ is said to be outgoing if it vanishes at infinity, and if $\Xb,\Yb$ defined in \eqref{3_XY3D} and $\beta$ defined in \eqref{3_ab} satisfy
\begin{equation}
\sqrt{R}\left[\partial_r-ik_n\right]\langle\Xb,\Yb_n\rangle(Re^{i\theta})\,,\,\sqrt{R}\left[\partial_r-ik_n\right]\langle\Yb,\Xb_n\rangle(Re^{i\theta})\underset{R\rightarrow +\infty}{\longrightarrow} 0 \quad \forall n>0 \quad \forall \theta\in (0,2\pi), 
\end{equation}
\begin{equation} \sqrt{R}\left[\partial_r-i\kappa_n\right]\left(\beta \, | \, \varphi_n\right)(Re^{i\theta}) \underset{R \rightarrow +\infty}{\longrightarrow} 0 \quad \forall n\geq 0 \quad \forall \theta \in (0,2\pi). \end{equation} 
\end{defi}

Under this condition, uniqueness of solutions to the source problem in 3D
will be guaranteed, as stated in the next Theorem.
We first introduce some notations. We define the scalar convolution~by 
\begin{equation}
\bm{g_1}\ast\cdot \bm{g_2}(x)=\int_{\R^2} \bm{g_1}(\xb-\bm{y})\cdot \bm{g_2}(\bm{y})\dd \bm{y},
\end{equation}
and introduce the Green functions
\begin{equation}
G_1^n(\xb)=-\frac{i}{4}\hau_0(k_n|\xb|), \quad \bm{G_2}^n(\xb)=\nabla G_1^n(\xb), \quad G_3^n(\xb)=-\frac{i}{4}\hau_0(\kappa_n|\xb|). 
\end{equation} 
Let $(\bm{f}^L,\bm{f}^{sh})$, $(\bm{b_\top}^L,\bm{b_\top}^{sh})$, $(\bm{b_\bot}^L,\bm{b_\bot}^{sh})$ denote the HHD of $(f_1,f_2)$, $(b_1^\top,b_2^\top)$, $(b_1^\bot,b_2^\bot)$ respectively and set
\begin{equation}
g_n^z(\xb)=\frac{1}{J_n}\left(\int_{-h}^h f_3(\xb,z)v_n(z)\dd z +b_3^\top(\xb) v_n(h)+b_3^\bot(\xb) v_n(-h)\right), 
\end{equation}
\begin{equation}
\bm{g}_n^\ell=\frac{1}{J_n}\left(\int_{-h}^h \bm{f}^\ell(\xb,z)u_n(z)\dd z +\bm{b}_\top^\ell u_n(h)+\bm{b}_\bot^{\ell} u_n(-h)\right), \quad \ell \in \{L,sh\}.
\end{equation}

%%-------------------------------------------------
\begin{theorem}\label{3_solution3D}
Let $r>0$. For every $\omega\notin \omega_{\text{crit}}\cup \omega_{\text{crit}}^{sh}$, $\bm{f}\in \text{H}^3(\Omega_r)$ and $\bb^\top,\bb^\bot\in \widetilde{H}^{3/2}(-r,r)$, the problem 
\begin{equation}\label{3_eq3D}
\left\{\begin{array}{cl} \nabla\cdot\sib(\ub)+\omega^2\ub=-\bm{f} & \text{ in } \Omega, \\ \sib(\ub)\cdot \nu= \bb^{\top/\bot} & \text{ on } \Omega_{\top/\bot}, \\
\ub \text{ is outgoing,} &\end{array}\right. 
\end{equation}
has a unique solution $\ub\in \text{H}^3_\loc(\Omega)$. This solution admits a decomposition $\ub=\ub^L+\ub^{sh}$ with
\begin{equation}\label{3_sol3D}
\ub^L(\bm{x},z)=\left(\begin{array}{c} \sum_{n>0} \bm{A}_n(\bm{x})u_n(z) \\ \sum_{n>0} b_n(\bm{x}) v_n(z) \end{array}\right), \quad \ub^{Sh}(\bm{x},z)=\left(\begin{array}{c} \sum_{n\geq 0} \bm{C}_n(\bm{x})\varphi_n(z) \\ 0 \end{array}\right),
\end{equation}
where $\bm{A}_n, b_n, \bm{C}_n$ satisfy the equations 
\begin{equation}
\left\{\begin{array}{c} \Delta_2 \bm{A}_n +k_n^2 \bm{A}_n=-\nabla g_n^z+ik_n\bm{g}^L_n, \\ \Delta_2 \bm{C}_n +\kappa_n^2 \bm{C}_n = -\bm{g}^{sh}_n, \\ \Delta_2 b_n +k_n^2 b_n = -ik_ng_n^z + \text{div}_2(\bm{g}^L_n). 
\end{array}\right. 
\end{equation}
Equivalently, $\bm{A}_n=-g_n^z\ast \bm{G_2}^n+ik_n\bm{g}_n^L\ast G_1^n$, 
$\bm{C}_n=-\bm{g}_n^{sh}/\mu\ast G_3^n$ and 
$b_n=-ik_ng_n^z\ast G^n_1+\bm{g}^L_n\ast\cdot\,\bm{G_2}^n$.
Moreover, there exists a constant $C>0$ depending only on $h$, $\omega$ and $r$ such that 
\begin{equation}
\Vert \ub \Vert_{\text{H}^3(\Omega_r)}\leq C \left(\Vert \fb \Vert_{\text{H}^1(\Omega)}+\Vert \bb^\top \Vert_{\text{H}^{3/2}(\R)}+\Vert \bb^\bot \Vert_{\text{H}^{3/2}(\R)}\right). 
\end{equation}
\end{theorem}

\begin{proof}
If $\fb$, $\bb^\top$ and $\bb^\bot$ vanish, uniqueness in Propositions \ref{3_prop3D1} and \ref{3_prop3D2} show that $\alpha=\beta=w=0$. Since $u$ and $v$ are uniquely determined by $\di_2(u,v)$ and $\curl_2(u,v)$, it follows that $u=v=0$ and the uniqueness of a solution is established. Moreover, Propositions \ref{3_prop3D1} and \ref{3_prop3D2} provide expressions of $\alpha$, $\beta$ and $w$. Using the HHD, there exists a unique wavefield $\ub$ determined by $(\alpha,\beta)$ and we can
check that expressions provided in \eqref{3_sol3D} indeed provide a solution. 
Finally, the control of the wavefield with respect to source terms is obtained 
in the same manner as in the proof of Theorem~\ref{3_solution2D}, using the following 
estimates on the Green functions: 
\begin{equation}
\Vert G_1^n\Vert_{\text{L}^1(B(0,r))}=\cO(1/n^2), 
\quad \Vert \bm{G_2}^n\Vert_{\text{L}^1(B(0,r))}=\cO(1/n).
\end{equation}
\end{proof}

%%-------------------------------------------------------
\section{Reconstruction of small shape defects from multi-frequency measurements in 2D}
%%-------------------------------------------------------

In this section, we consider the inverse problem of reconstructing of small shape 
defects in an elastic plate from multi-frequency surface measurements.
We first detail the method used in the two dimensional case. Its 3D generalization
is discussed at the end of the section.
We follow the method developed in \cite{bonnetier1} for acoustic waveguides. 
In the present case, given current experimental setups~\cite{legrand1},
we assume that the measurements consist in surface measurements of the displacement fields, 
rather than measurements in a section of the waveguide, which were considered in
the acoustic case.

We consider a plate $\wt \Omega$ that contains localized bumps, defined by
\begin{eqnarray*}
\wt \Omega &=& \{ (x,z) \in \R^2 \,|\, h(-1 + 2g_2(x)) < z < h(1 + 2g_1(x)) \},
\end{eqnarray*}
where $g_1, g_2$ are ${\mathcal C^2}$ functions with compact support, such that 
$-1 + 2g_2(x) < 1 + 2g_1(x)$, see Figure~\ref{3_shapedef}. 
Note that $g_1$ and $g_2$ are not required to have 
a constant sign. Hereafter $\wt \Omega$ is called the perturbed plate.

\begin{figure}[h]

\begin{center}
\begin{tikzpicture}
\draw (0,1) -- (3,1);
\draw (5,1) -- (8,1);
\draw [dashed] (-1,1) -- (0,1);
\draw [dashed] (-1,0) -- (0,0);
\draw (-1,0) node[left]{$-h$}; 
\draw (-1,1) node[left]{$h$}; 
\draw [dashed] (3,1)--(5,1);
\draw (0,0) -- (2,0);
\draw (7,0) -- (8,0);
\draw [dashed] (2,0) -- (7,0);
\draw [dashed] (8,0) -- (9,0);
\draw [dashed] (8,1) -- (9,1);
\draw [domain=3:5] [samples=200] plot (\x,{5/16*(\x-3)^2*(\x-5)^2+1});
\draw [domain=2:7] [samples=200] plot (\x,{1/4*sin(2*(\x-2)*3.14/5 r)});
\draw [->] [domain=0:2.5] [samples=200] plot (\x,{0.2*sin(10*(\x+0.6) r)+0.5});
\draw (-0.3,0.3) node[above]{$\ub^{\text{inc}}$};
\draw[<->] (4,1) -- (4, {5/16*(4-3)^2*(4-5)^2+1});
\draw[<->] (3,0) -- (3, {4/16*(3-2)^2*(3-4)^2});
\draw (4.2,0.7) node{$2hg_1(x)$};
\draw (3,-0.3) node{$2hg_2(x)$};
\draw (6,0.5) node{$\wt \Omega$}; 

\end{tikzpicture} \caption{\label{3_shapedef} Representation of shape defects in a plate of width $h$.}\end{center}\end{figure}
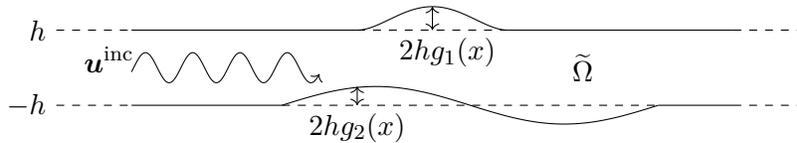

An incident wavefield $\ub^\text{inc}$ is send in the waveguide, and we denote by $\wt \ub$ the total wavefield and $\wt\ub^s:= \wt\ub-\ub^{\text{inc}}$ the scattered wavefield. Our goal is to reconstruct functions $g_1$ and $g_2$ from the wavefields
scattered by the defects, and to solve the inverse problem 
\begin{equation}
\text{Find } \quad (g_1,g_2) \quad \text{from } \quad \wt \ub^s(\omega,x,h(x)) \qquad \forall x\in \R \quad \forall \omega\in (0,\omega_{\max}).
\end{equation}

\subsection{Born approximation}

As incident wave, we use the function corresponding to
the first symmetric Lamb mode of a straight guide, which
we denote by $\ub^{\text{inc}}(x,z):=\ub_1(z)e^{ik_1x}$. The same analysis could be conducted with any other Lamb mode. However, note that the first symmetric Lamb mode has the advantage to propagate at any frequency, which is not the case for the other symmetric Lamb modes. The total wavefield $\wt \ub$ solves the equations of elasticity in the waveguide:
\begin{equation}
\left\{\begin{array}{cl} \nabla\cdot\sib(\wt\ub)+\omega^2\wt\ub=0 & \text{ in } \wt\Omega, \\ \sib(\wt\ub)\cdot \nu= 0 & \text{ on } \partial\wt\Omega_{\top/\bot}.\end{array}\right. 
\end{equation}
Then, the scattered wavefield solves
\begin{equation}\label{3_scattered1}
\left\{\begin{array}{cl} \nabla\cdot \sib(\wt\ub^s)+\omega^2\wt\ub^s=0 & \text{ in } \wt\Omega, \\ \sib(\wt\ub^s)\cdot \nu= -\sib(\ub^{\text{inc}})\cdot \nu & \text{ on } \partial\wt\Omega_{\top/\bot}, 
\\
\ub^s \text{ is outgoing.} \end{array}\right.
\end{equation}

For every $(x,z)\in\partial\wt\Omega$, we know, using the definition of Lamb modes, that
\begin{equation}
\sib(\ub^{\text{inc}})=\left(\begin{array}{cc} s_1(z) & t_1(z) \\ t_1(z) & r_1(z)\end{array}\right)e^{ik_1x},
\end{equation}
so if $(x,z)\in\partial \wt\Omega_\top$ then $\sib(\wt \ub^s)\cdot \nu$ is equal to 
\begin{equation}
-\sib({\wt \ub}^\text{inc})\cdot \frac{1}{\sqrt{1+4h^2g_1'(x)^2}}\left(\begin{array}{c} -2hg_1'(x) \\ 1\end{array}\right)=-\frac{e^{ik_1x}}{\sqrt{1+4h^2g_1'(x)^2}} \left(\begin{array}{c} -2hg_1'(x)s_1(z)+t_1(z) \\ -2hg_1'(x)t_1(z)+r_1(z) \end{array}\right).
\end{equation}
We can also do the same thing on $\partial\wt \Omega_\bot$ to explicit the equation of the scattered wavefield. Then, following the steps of \cite{bonnetier1}, we map the perturbed waveguide $\wt \Omega$ to a regular waveguide $\Omega:=\R\times (-h,h)$ using the mapping 
\begin{equation}
\phi(x,z)=\left(x,\left(1+g_1(x)-g_2(x)\right)z+hg_1(x)+hg_2(x)\right). 
\end{equation}
We define $\ub^s=\wt \ub^s \circ \phi$. Equation \eqref{3_scattered1} in the perturbed waveguide $\wt \Omega$ is equivalent to the following equation in the regular waveguide $\Omega$:

\begin{equation}\label{3_avantborn}
\left\{\begin{array}{cl} \nabla\cdot \sib(\ub^s)+\omega^2\ub^s=-a(\ub^s) & \text{ in } \Omega, \\ \sib(\ub^s)\cdot \nu= \left(\begin{array}{c} 2hg_1'(x)s_1(\phi(x,z))-t_1(\phi(x,z)) \\ 2hg_1'(x)t_1(\phi(x,z))-r_1(\phi(x,z)) \end{array}\right)e^{ik_1x} -b_1(\ub^s) & \text{ on } \partial\Omega_{\top}, \\
\sib(\ub^s)\cdot \nu= \left(\begin{array}{c} -2hg_2'(x)s_1(\phi(x,z))-t_1(\phi(x,z)) \\ -2hg_2'(x)t_1(\phi(x,z))-r_1(\phi(x,z)) \end{array}\right)e^{ik_1x}-b_2(\ub^s)& \text{ on } \partial\Omega_{\bot}, \\
\ub^s \text{ is outgoing,} \end{array}\right.
\end{equation}
where we denote $f_1(x)=hg_1(x)+hg_2(x)$, $f_2(x)=1+g_1(x)-g_2(x)$, 

\begin{multline}
a(\ub)=-\frac{f_1'f_2+f_2'(z-f_1)}{{f_2}^2}\left(\begin{array}{c} 2(\lambda+2\mu)\partial_{xz} u+(\lambda+\mu)\partial_{zz} v \\ 2\mu \partial_{xz} v +(\lambda+\mu)\partial_{zz} u \end{array}\right) -\frac{f_2'}{{f_2}^2}\left(\begin{array}{c} (\lambda+2\mu)\partial_z v\\ \mu \partial_z u\end{array}\right) \\ 
+\frac{(f_1'f_2+f_2'(z-f_1))^2}{{f_2}^4}\left(\begin{array}{c} (\lambda+2\mu)\partial_{zz} u \\ \mu \partial_{zz} v \end{array}\right)  +\left(\frac{1}{f_2}-1\right)\left(\begin{array}{c} (\lambda+\mu)\partial_{xz} v \\ (\lambda+\mu) \partial_{xz} u \end{array}\right) \\ 
 -\frac{f_1''{f_2}^2+(f_2''-2{f_2'}^2f_2)(z-f_1)-2f_2'f_1'f_2}{{f_2}^3}\left(\begin{array}{c} (\lambda+2\mu) \partial_y u \\ \mu \partial_z v \end{array}\right) 
  +\left(\frac{1}{{f_2}^2}-1\right)\left(\begin{array}{c} \mu \partial_{zz} u \\ (\lambda+2\mu) \partial_{zz} v \end{array}\right),
\end{multline}
and 
\begin{equation}
b_1(\ub)=-g_1'(\lambda+2\mu)\partial_x u -\lambda\frac{g_1'}{f_2}\partial_z v+\mu\left(\frac{1}{f_2}-1\right)\partial_z u+\frac{f_1'f_2+f_2'(z-f_1)}{{f_2}^2}(g_1'(\lambda+2\mu)\partial_z u-\mu \partial_z v),
\end{equation}
\begin{equation}
b_2(\ub)=g_2'\mu\partial_x v +\mu\frac{g_2'}{f_2}\partial_z u-(\lambda+2\mu)\left(\frac{1}{f_2}-1\right)\partial_z v-\frac{f_1'f_2+f_2'(z-f_1)}{{f_2}^2}(g_2'\mu\partial_z v-\lambda \partial_z u).
\end{equation}

From now on, we only consider small shape defects, i.e. we assume that the quantity 
\begin{equation}
\eps=\max\left(\Vert g_1\Vert_{\text{W}^{2,\infty}(\R)}, \Vert g_2\Vert_{\text{W}^{2,\infty}(\R)}\right),\end{equation}
is small compared to the size of the supports of $g_1$ and $g_2$, and compared to the width of the waveguide. A direct computation leads to the following bounds for operators $a$ and $b$: 
\begin{prop}
For every $r>0$, there exist two constants $A,B>0$ depending only on $\omega$, $h$ and $r$ such that 
\begin{equation}
\Vert a(\ub)\Vert_{\text{H}^1(\Omega_r)}\leq A\eps \Vert \ub \Vert_{\text{H}^3(\Omega_r)}, \quad \Vert b(\ub)\Vert_{\text{H}^{3/2}(\Omega_r)}\leq B\eps \Vert \ub \Vert_{\text{H}^3(\Omega_r)}.
\end{equation}
\end{prop}

Following the steps of \cite{bonnetier1}, we define the Born approximation $\bm{v}$ of $\ub^s$ by
\begin{equation}\label{3_apresbornb}
\left\{\begin{array}{cl} \nabla\cdot \sib(\bm{v})+\omega^2\bm{v}=0 & \text{ in } \Omega, \\ \sib(\bm{v})\cdot \nu=\left(\begin{array}{c} 2hg_1'(x)s_1(\phi(x,z))-t_1(\phi(x,z)) \\ 2hg_1'(x)t_1(\phi(x,z))-r_1(\phi(x,z)) \end{array}\right)e^{ik_1x}  & \text{ on } \partial\Omega_{\top}, \\
\sib(\bm{v})\cdot \nu= \left(\begin{array}{c} -2hg_2'(x)s_1(\phi(x,z))-t_1(\phi(x,z)) \\ -2hg_2'(x)t_1(\phi(x,z))-r_1(\phi(x,z)) \end{array}\right)e^{ik_1x} & \text{ on } \partial\Omega_{\bot}, \\
\bm{v} \text{ is outgoing.} \end{array}\right.
\end{equation}
The following proposition, 
the proof of which is similar to Propositions 5 and 6 of \cite{bonnetier1},
shows that $\bm{v}$ is a good approximation of $\ub$ if the defect is small: 
\begin{prop}
Let $C>0$ be the constant defined in Theorem \ref{3_solution2D}. If $\eps C(A+B)<1$ then \eqref{3_avantborn} has a unique solution $\ub^s$ and 
\begin{equation}\label{3_born}
\Vert \ub^s -\bm{v}\Vert_{\text{H}^3(\Omega_r)}\leq \frac{\eps C(A+B)}{1-\eps C(A+B)}4rCh \eps (\Vert s_1\Vert_{\text{H}^2}+\Vert t_1\Vert_{\text{H}^2}+\Vert r_1\Vert_{\text{H}^2}). 
\end{equation}
\end{prop}

Finally, to simplify the boundary source term and get rid of the dependency on $\phi$, we notice that 
\begin{equation}
g_1'(x)s_1(\phi(x,z))=g_1'(x)s_1(h)+\cO(\eps^2), \quad t_1(\phi(x,z))=(g_1'(x)-g_2'(x))\partial_zt_1(h) +\cO(\eps^2), \end{equation}
\begin{equation}
g_1'(x)t_1(\phi(x,z))=\cO(\eps^2),\quad  r_1(\phi(x,z))=(g_1'(x)-g_2'(x))\partial_zr_1(h)+\cO(\eps^2).
\end{equation}
We define a simpler approximation $\bm{w}$ of $\bm{v}$ as the solution of
\begin{equation}\label{3_apresborn}
\left\{\begin{array}{cl} \nabla\cdot \sib(\bm{w})+\omega^2\bm{w}=0 & \text{ in } \Omega, \\ \sib(\bm{w})\cdot \nu=\left(\begin{array}{c} 2hg_1'(x)s_1(h)-(g_1'(x)-g_2'(x))\partial_zt_1(h) \\ -(g_1'(x)-g_2'(x))\partial_zr_1(h) \end{array}\right)e^{ik_1x}  & \text{ on } \partial\Omega_{\top}, \\
\sib(\bm{w})\cdot \nu= \left(\begin{array}{c} -2hg_2'(x)s_1(h)-(g_1'(x)-g_2'(x))\partial_zt_1(h) \\ (g_1'(x)-g_2'(x))\partial_zr_1(h) \end{array}\right)e^{ik_1x} & \text{ on } \partial\Omega_{\bot}, \\
\bm{w} \text{ is outgoing.} \end{array}\right.
\end{equation}

Using the control provided by Theorem \ref{3_th2}, $\bm{w}$ is a good approximation of $\bm{v}$ if $\eps$ is small enough and there exists a constant $D>0$ such that 
\begin{equation}
\Vert \bm{v}-\bm{w}\Vert_{\text{H}^3(\Omega_r)}\leq \eps^2 D h r \left(\Vert s_1\Vert_{\text{H}^2}+\Vert \partial_z t_1\Vert_{\text{H}^2}+\Vert \partial_z r_1\Vert_{\text{H}^2}\right).
\end{equation}

%%--------------------------------------------------------

\subsection{Boundary source inversion}

From now on, we denote by $\ub$ the solution to \eqref{3_apresborn} generated 
with boundary source terms denoted by $\bb^\top$ and $\bb^\bot$. Given a maximal frequency $\omega_{\max}$, we measure the wavefield at the surface of the perturbed plate for every $\omega\in (0,\omega_{\max})$. Using the previous Born approximation, we can assume that the wavefield $\ub$ is measured on the surface $y=h$ and that the measurements may contain noise.
For every frequency $\omega$ and $x\in \R$, the measured value of $\ub(x,h)$ is denoted by
$\ub_\omega(x,h)$. Similarly, the associated wavenumbers and Lamb modes are denoted by $k_n(\omega)$ and $(u_{n,\omega},v_{n,\omega})$ respectively. 
Using Theorem \ref{3_solution2D}, we know that 
\begin{equation}
u_\omega(x,h)=\sum_{n>0}a_n(x)u_{n,\omega}(h), \qquad v_\omega(x,h)=\sum_{n>0} b_n(x)v_{n,\omega}(h), 
\end{equation}
where $a_n=G_1^n\ast F^n_1-G_2^n\ast F_2^n$, $b_n=G_2^n\ast F_1^n-G_1^n\ast F_2^n$ and $G_1^n, G_2^n$ are defined in \eqref{3_gi} and 
\begin{equation}\label{3_F1b}
F^n_1(x)=\frac{e^{ik_1x}}{J_n}\left((2hg_1's_1(h)-(g_1'-g_2')\partial_zt_1(h))u_{n,\omega}(h)-(2hg_2's_1(h)+(g_1'-g_2')\partial_zt_1(h))u_{n,\omega}(-h)\right),\end{equation}
\begin{equation}\label{3_F2b}
 F^n_2=\frac{e^{ik_1x}}{J_n}(g_1'-g_2')\partial_zr_1(g)\left(-u_{n,\omega}(h)+u_{n,\omega}(-h)\right). 
\end{equation}

Assuming that $x$ is located on the left of the support of the sources, 
\begin{equation}
(u_\omega,v_\omega)(x,h)=\frac{1}{2}\sum_{n>0}(u_{n,\omega},-v_{n,\omega})(h)e^{-ik_n(\omega)x}\int_\R e^{ik_n(\omega)z}(F^n_1(z)-F^n_2(z))\dd z, \end{equation}
As explained in \cite{mallat1} and illustrated in \cite{legrand1}, we can use a spatial Fourier  transform along $x$ to separate each term of the sum.  We notice that up to a multiplicative coefficient, the fields $u_\omega$ and $v_\omega$ contain the same information about the source, so that only measurements of one component of the displacement are needed. 

Further, since noise is likely to pollute the response of evanescent and inhomogeneous modes
in real-life experiments, we only consider the propagative modes and for these modes
$n$ we have access to  
\begin{equation}\label{3_extractfourier}
 \int_\R e^{i(k_n(\omega)+k_1(\omega))z}(g_1'(z)c^1_n+g_2'(z)c^2_n)\dd z   \qquad \forall \omega\in \R_+,
\end{equation}
where $c^1_n$ and $c^2_n$ are known coefficients depending on the mode $n$. We use the following definition for the Fourier transform
\begin{equation}
\F(g)(\xi)=\int_\R g(z)e^{-i\xi z}\dd z.
\end{equation}
From now on, we consider that $n=1$ is the first propagative symmetric Lamb mode and $n=2$ is the first propagative antisymmetric Lamb mode. 
Both modes exist at any frequency $\omega$, and $\omega\mapsto k_1(\omega)$ or $\omega\mapsto k_2(\omega)$ are increasing functions that map $\R_+$ to $\R_+$ (see an illustration in Figure \ref{3_disp3D} for the symmetric case, and for more details we refer to \cite{royer1}). 
In particular, if we set $\xi=2k_1$, the available information amounts to knowing 
$\F(c_1^1g_1'+c_1^2g_2')(\xi)$ for every $\xi\in (0,2k_1(\omega_{\max}))$. 
Similarly, if $\xi=k_1+k_2$, we have knowledge of $\F(c_2^1g_1'+c_2^2g_2')$ for every 
$\xi \in (0,k_1(\omega_{\max})+k_2(\omega_{\max}))$. 
We define 
\begin{equation}
\xi_{\max}=\min\left(k_1(\omega_{\max})+k_2(\omega_{\max}), 2k_1(\omega_{\max})\right). 
\end{equation}

Looking at expressions \eqref{3_F1b}-\eqref{3_F2b}, we notice that the linear combinations $c_1^1g_1'+c_1^2g_2'$ and $c_2^1g_1'+c_2^2g_2'$ are independent so the functions $g_1'$ and $g_2'$ can be reconstructed using the inverse Fourier transform, in a stable way as the next Proposition shows 
(its proof is the same as Proposition 12 in \cite{bonnetier1}).

\begin{prop}
Let $g,g^\app\in \mathcal{C}^2(-r,r)$ and their Fourier transform $d=\F(g)$ and $d^\app=\F(g^{\app})$ defined on $(0,\xi_{\max})$. Assume that there exists $M>0$ such that $\Vert g\Vert_{\text{H}^1(-r,r)}, \Vert g^{\app}\Vert_{\text{H}^1(-r,r)}\leq M$, then 
\begin{equation}
\Vert g-g^\app\Vert^2_{\text{L}^2(-r,r)}\leq \frac{4}{\pi}\Vert \F(g)-\F(g^\app)\Vert_{\text{L}^2(0,\xi_{\max})}+\frac{2\pi}{\xi_{\max}^2}M^2. 
\end{equation}
\end{prop}

\begin{rem}
We notice that the above estimate is actually better than the one presented in the acoustic case in Proposition 12 of \cite{bonnetier1}. 
Indeed,  in the acoustic case, there is only one propagative mode at every frequency, and the function $k_1+k_2$ is not one-to-one from $\R_+$ to $\R_+$. 
In the elastic case however, we take advantage of the existence of two different Lamb modes propagating at every frequency. 
\end{rem}

Given the reconstructions of $g_1'$ and $g_2'$, we can integrate these functions,
using the fact that $g_1$ and $g_2$ have compact support, and obtain an approximation of 
the shapes of the defects. In the previous estimate, 
the error $\Vert d-d^{\app}\Vert_{\text{L}^2(-r,r)}$ contains both the measurement error, 
as well as the error caused by the Born approximation \eqref{3_born}. 
It follows that the reconstruction error decreases when the size of the defects gets 
smaller and when $\omega_{\max}$ increases. We present examples of numerical reconstructions 
in the next section.

We conclude this section by discussing possible extensions of this work.
First, the method presented here could be implemented in a similar fashion in 3D. 
Indeed, using a Born approximation, one can show that the measurements are close to 
those emanating from a wavefield generated by two boundary source terms
that depend on $\nabla g_1$ and $\nabla g_2$ in a regular waveguide. 
In 3D, a Hankel transform plays the role of the Fourier transform,
and one obtains thus a reconstruction.

Second, by the same method, one can also reconstruct bends in an elastic waveguide, 
in a similar manner as in the acoustic case described in section 3.2 of \cite{bonnetier1}.
However, the detection of homogeneities seems more difficult. 
Following section 3.4 of \cite{bonnetier1}, one could use a Born approximation to approximate
the measurements by a wavefield generated by an internal source term $\fb$,
that depends on a transformed inhomogeneity in a regular waveguide. 
However, it does not seem easy to extract, from the measurements, something like the Fourier transform of a function, that would characterize the inhomogeneity,
as in \eqref{3_extractfourier}.

%%------------------------------------------------------------

\section{Numerical results}

In this last section, we illustrate the results of Theorems \ref{3_solution2D} and \ref{3_solution3D}, and present numerical reconstructions of small shape defects. 

Concerning Theorems \ref{3_solution2D} and \ref{3_solution3D}, we compare the modal expressions of $\ub$ given in \eqref{3_sol2D} and \eqref{3_sol3D} to the wavefields generated using Matlab in 2D, 
and Freefem++~\cite{hecht1} in 3D, 
respectively used to solve \eqref{3_eq2D} and \eqref{3_eq3D}.
In the following, we assume that sources are supported in $\Omega_r$ where $r=3$ in 2D and $r=1$ in 3D. To solve the elastic equation, we use the finite element method with a perfectly matched layer (PML)~\cite{berenger1} placed in $\Omega_8\setminus \Omega_4$ in 2D, and $\Omega_{2.5}\setminus \Omega_{1.3}$. 
Since PML's do not handle the presence of right-going propagating modes correctly when
the wavenumber is negative (see an example of such wavenumbers in Figure \ref{3_rgmode}), 
we use the strategy presented in \cite{bonnet3} which modify the PML to provide a
correct approximation of the wavefield, for every non critical frequency. 
The coefficient of absorption in the PML is defined by $\alpha=-((|x|-4)\textbf{1}_{|x|\geq 4}$ 
in 2D or $\alpha=-((|\xb|-1.3)\textbf{1}_{|\xb|\geq 1.3}$ in 3D. 
The structured mesh is built with a stepsize of $10^{-4}$ in 2D and $10^{-2}$ in 3D. 

We first illustrate the two dimensional case, and the modal decomposition \eqref{3_sol2D} solution to \eqref{3_eq2D}.
Numerical representations of the wavefield $\ub$, obtained using the modal decomposition \eqref{3_sol2D} are presented in Figure \ref{3_num2D} as well as the wavefields generated
by the finite element method, showing good visual agreement. Their computed relative error 
in $\text{L}^\infty(\Omega_r)$ and $\text{L}^2(\Omega_r)$ is smaller than $2\%$. 

\begin{figure}[h]\begin{center}
\begin{tabular}{cc}
\begin{tikzpicture}[scale=1]
\begin{axis}[width=4.8cm, height=1.5cm, axis on top, scale only axis, xmin=-3, xmax=3, ymin=-0.1, ymax=0.1, title={$\Real(u^{\text{mod}})$},axis line style={draw=none},tick style={draw=none},ytick distance=0.1]
\addplot graphics [xmin=-3,xmax=3,ymin=-0.1,ymax=0.1]{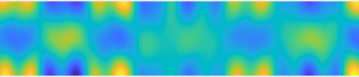};
\input{source2D}
\end{axis} 
\end{tikzpicture} &
\begin{tikzpicture}[scale=1]
\begin{axis}[width=4.8cm, height=1.5cm, axis on top, scale only axis, xmin=-3, xmax=3, ymin=-0.1, ymax=0.1, colorbar,point meta min=-2.1302    ,point meta max=2.3249, title={$\Real(u^{\text{FEM}})$},axis line style={draw=none},tick style={draw=none},ytick distance=0.1]
\addplot graphics [xmin=-3,xmax=3,ymin=-0.1,ymax=0.1]{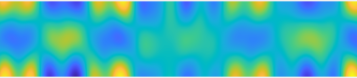};
\input{source2D}
\end{axis} 
\end{tikzpicture}
\\
\begin{tikzpicture}[scale=1]
\begin{axis}[width=4.8cm, height=1.5cm, axis on top, scale only axis, xmin=-3, xmax=3, ymin=-0.1, ymax=0.1, title={$\Imag(v^{\text{mod}})$},axis line style={draw=none},tick style={draw=none},ytick distance=0.1]
\addplot graphics [xmin=-3,xmax=3,ymin=-0.1,ymax=0.1]{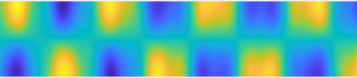};
\input{source2D}
\end{axis} 
\end{tikzpicture} &
\begin{tikzpicture}[scale=1]
\begin{axis}[width=4.8cm, height=1.5cm, axis on top, scale only axis, xmin=-3, xmax=3, ymin=-0.1, ymax=0.1, colorbar,point meta min= -3.2517    ,point meta max=3.2783 , title={$\Imag(v^{\text{FEM}})$},axis line style={draw=none},tick style={draw=none},ytick distance=0.1]
\addplot graphics [xmin=-3,xmax=3,ymin=-0.1,ymax=0.1]{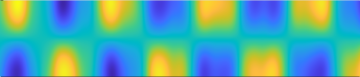};
\input{source2D}
\end{axis} 
\end{tikzpicture}\end{tabular}\end{center}
\caption{\label{3_num2D} Comparison between a wavefield $\ub$ computed using the modal solution \eqref{3_sol2D} or using a finite element method. Top: comparison between real parts of $u$. Bottom: comparison between imaginary parts of $v$. Similar results could also be obtained for $\Imag(u)$ and $\Real(v)$. The parameters of the problem are $\lambda=0.31$, $\mu=0.25$, $h=0.1$, $\omega=13.7$. The sum in the modal decomposition of $\ub^{\text{mod}}$ is cut at $N=20$ modes. Wavefields are generated using an internal source term $\fb$ defined in \eqref{3_fb2D} and boundary source terms $\bb^\top$ and $\bb^\bot$ defined in \eqref{3_bb2D}. Their support is represented in red. Here, the relative $\text{L}^\infty(\Omega_r)$-error is $1.7\%$ and the $\text{L}^2(\Omega_r)$-error is $1.4\%$.}
\end{figure}

A similar comparison is carried out in 3D for the modal decomposition \eqref{3_sol3D} solution to \eqref{3_eq3D}. To visualize the decomposition $\ub=\ub^L+\ub^{sh}$, we first choose a curl-free internal source given by \eqref{3_fb3D}.
The modal simulation is compared with the fields obtained from a finite element approximation 
in Figure \ref{3_num3D1}. Again, both approximations of the true wavefield are visually similar, even if that produced by the finite element discretization seems to propagate at a higher velocity. 
This could be caused by the fact that the step size of the discretization may not be sufficiently small. We point out that the calculation times of these simulations are not the same: while the finite element method takes around eight hours to run, the modal decomposition produces a result in less than two minutes. This underlines the interest of using the modal solution to do computations in three-dimensional perfect plates. 
Next, we choose a divergence-free boundary source term given by \eqref{3_bb3D}. Comparisons are presented in Figure \ref{3_num3D2}, and similar conclusions can be drawn.

\begin{figure}[h]\begin{center}
\begin{tabular}{ccc} $\Real(w^{\text{mod}})$\hspace{2cm} & $\Real(w^{\text{FEM}})$\hspace{2cm} & \\
\includegraphics[width=5cm]{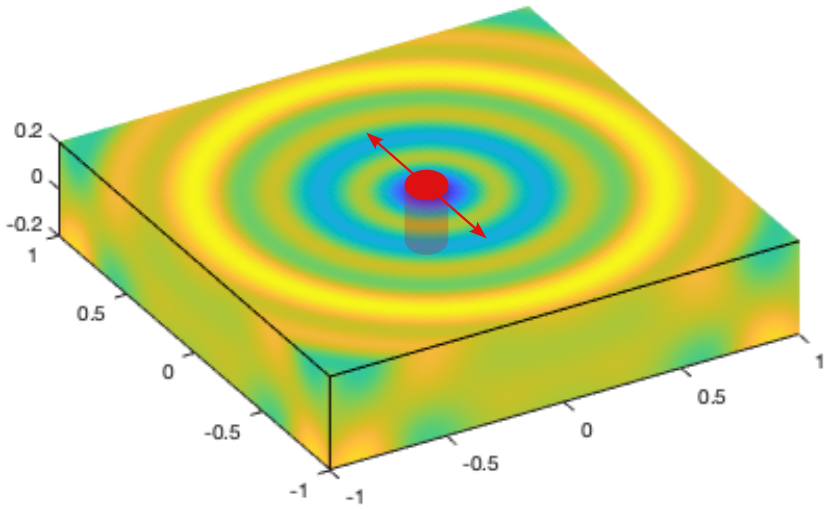} &\includegraphics[width=5cm]{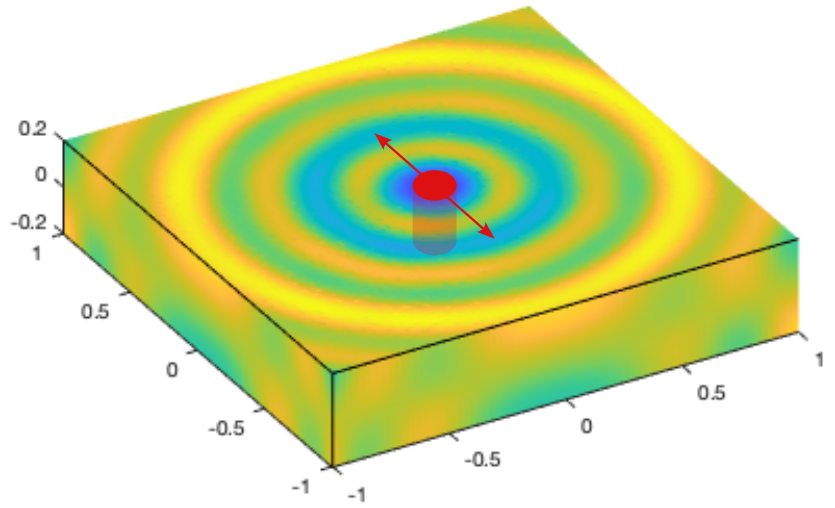} &\begin{tikzpicture}[scale=1]
\begin{axis}[width=0.1cm, height=2.5cm, hide axis, scale only axis, xmin=0, xmax=1, ymin=0, ymax=1, colorbar,point meta min= -0.0260  ,point meta max=   0.0096,tick style={draw=none}]
\end{axis} 
\end{tikzpicture}\\[3mm]
$\Real(u^{\text{mod}})$\hspace{2cm} & $\Real(u^{\text{FEM}})$\hspace{2cm} & \\
\includegraphics[width=5cm]{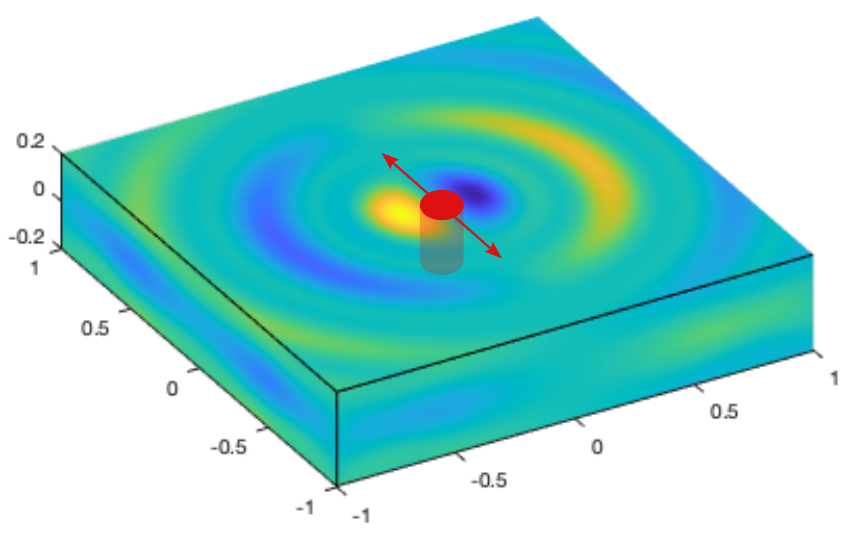} &\includegraphics[width=5cm]{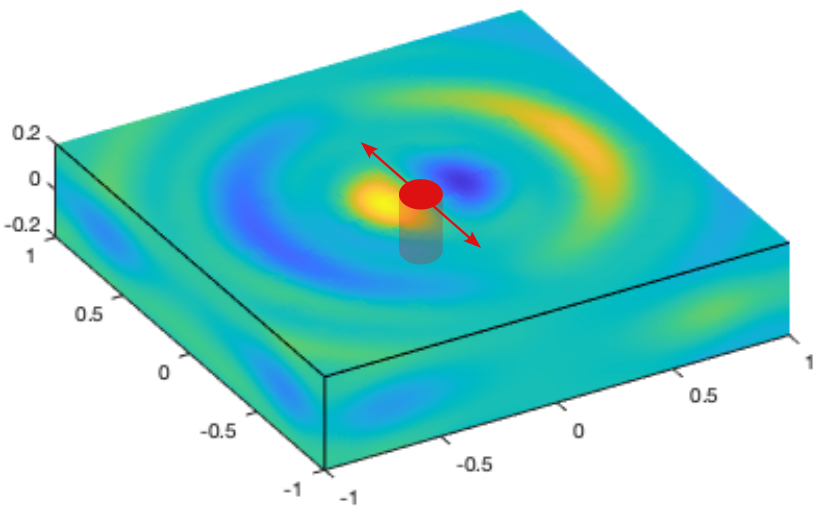}&
\begin{tikzpicture}[scale=1]
\begin{axis}[width=0.1cm, height=2.5cm, hide axis, scale only axis, xmin=0, xmax=1, ymin=0, ymax=1, colorbar,point meta min= - 0.0155 ,point meta max=    0.0155,tick style={draw=none}]
\end{axis} 
\end{tikzpicture}\end{tabular}
\caption{\label{3_num3D1} Comparison between $\ub$ computed using the modal solution \eqref{3_sol3D} and using the finite element method with a curl-free internal source \eqref{3_fb3D} represented in red. Top: comparison of 
$\Real(w)$. Bottom: comparison of $\Real(u)$ (plots of $\Imag(u),\Real(v),\Imag(v),\Imag(w)$
look similar). The parameters are $\lambda=0.31$, $\mu=0.25$, $h=0.2$, $\omega=10$. 
$N=20$ modes are used in the decomposition of $\ub^{\text{mod}}$. The relative $\text{L}^\infty(\Omega_r)$-error is $7.2\%$ and the $\text{L}^2(\Omega_r)$-error is $9.4\%$.}
 \end{center}
\end{figure}

\begin{figure}[h]\begin{center}
\begin{tabular}{ccc} $\Real(v^{\text{mod}})$\hspace{2cm} & $\Real(v^{\text{FEM}})$\hspace{2cm} & \\
\includegraphics[width=5cm]{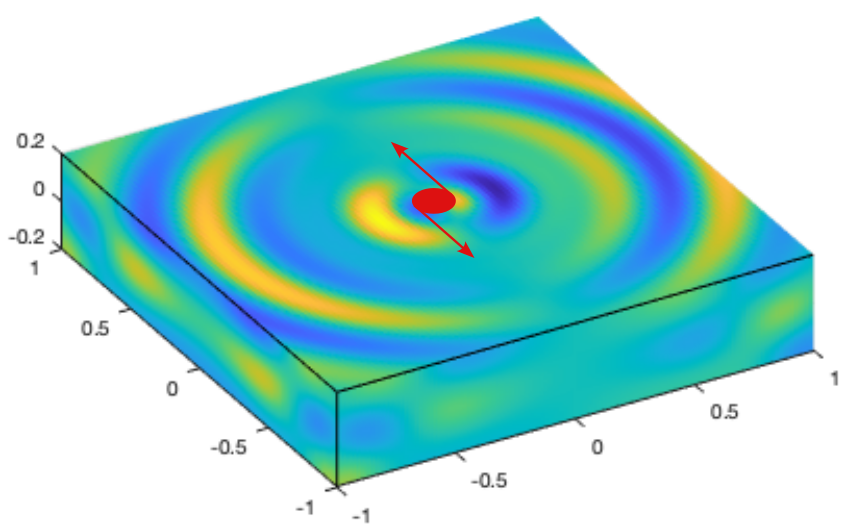} &\includegraphics[width=5cm]{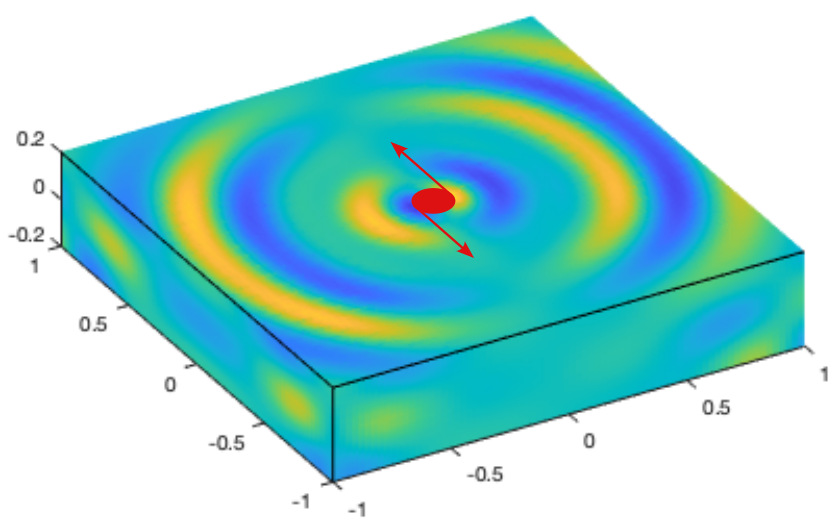} &\begin{tikzpicture}[scale=1]
\begin{axis}[width=0.1cm, height=2.5cm, hide axis, scale only axis, xmin=0, xmax=1, ymin=0, ymax=1, colorbar,point meta min=-0.0649 ,point meta max=    0.0649,tick style={draw=none}]
\end{axis} 
\end{tikzpicture}\end{tabular}
\caption{\label{3_num3D2}Comparison between a wavefield $\ub$ computed using the modal solution \eqref{3_sol3D} or using a finite element method with a divergence free boundary source \eqref{3_bb3D} represented in red. The real parts of $v$ is shown (plots of $\Imag(v),\Real(u),\Imag(u),\Real(w),\Imag(w)$
would have similar aspect). The parameters are $\lambda=0.31$, $\mu=0.25$, $h=0.2$, $\omega=10$. 
$N=10$ modes are used in the decomposition of $\ub^{\text{mod}}$. The relative $\text{L}^\infty(\Omega_r)$-error is $5.3\%$ and the $\text{L}^2(\Omega_r)$-error is $4.1\%$.}
 \end{center}
\end{figure}

Finally, we illustrate in Figure \ref{3_rec} two numerical reconstructions of small defects. Synthetic surface measurements are generated using the finite element method described above for different frequencies. Then, we reconstruct the derivative of defects profiles $g_1$ and $g_2$ using the penalized least square algorithm described in \cite{bonnetier1}. We get reconstructions as good or even better than the one presented in the acoustic case (see Figure 11 in \cite{bonnetier1}), and we notice that the reconstruction seems more robust than the acoustic one then the size of the defect increases. Table \ref{tableau2} illustrates this point as it depicts the relative error on a reconstruction of $g_1$ when its amplitude increases. This table can be compared to Table 2 in \cite{bonnetier1} where the same relative error in the acoustic case turns out to be bigger.

\begin{figure}[H]
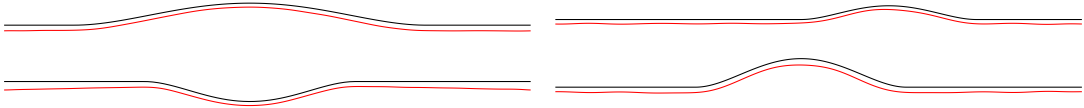

\begin{center}
\input{rec1}\hspace{2mm}
\input{rec2}
\end{center}
\caption{\label{3_rec} Reconstruction of two shape defects. In black, the initial shape of $\Omega$, and in red the reconstruction, slightly shifted for comparison purposes. In both cases, $h=0.1$, $\omega_{\max}=17$ and the interval $(0,\omega_{\max})$ is discretize with $170$ points. The relative $\text{L}^2$-error is $4.7\%$ on the left and $5.1\%$ on the right. Functions $g_1$ and $g_2$ are given in \eqref{g1g2_1} and \eqref{g1g2_2}.}
\end{figure}

\begin{table}[H]
\begin{center}
\begin{tabular}{|c|c|c|c|c|} \hline  $A$ & $0.1$ & $0.2$ &  $0.3$ & $0.5$ \\ \hline $\Vert g_1-g_{1}^\text{app}\Vert_{\text{L}^2(\R)}/\Vert g_1\Vert_{\text{L}^2(\R)}$ & $4.5\%$ & $6.4\%$  &$9.2\%$ & $18.3\%$\\ \hline \end{tabular}
\end{center}
\caption{\label{tableau2} Relative errors on the reconstruction of $h$ for different amplitudes $A$ for the shape defects $g_1(x)=A\textbf{1}_{3\leq x\leq 5}(x-3)^2(5-x)^2$ and $g_2(x)=0$. In every reconstruction, $h=0.1$, $\omega_{\max}=17$ and the interval $(0,\omega_{\max})$ is discretize with $170$ points.}
\end{table}

\section*{Appendix: Expressions for numerical simulations}

\begin{equation}\label{3_fb2D}\tag{E1} \fb(x,y)=-100\,\textbf{1}_{(x-0.5)^2+\frac{(y-0.06)^2}{0.015^2}<1}(x,y)\left(1-(x-0.5)^2+\frac{(y-0.06)^2}{0.015^2}\right)( x+2y; 1),\hfill
\end{equation}
\begin{equation}\label{3_bb2D}\tag{E2}
\bb^\top(x)=\frac{10}{\sqrt{2\pi}}e^{-\frac{(x+0.5)^2}{200}}(1; x), \quad \bb^\bot(x)=20\,\textbf{1}_{[2,2.5]}(x)(x-2)(x-2.5)(1; \sin(x)),\hfill
\end{equation}
\begin{equation}\label{3_fb3D} \tag{E3}
\fb(x,y,z)=z\frac{50}{\pi}e^{-\frac{(x^2+y^2)}{200}}\left( -x; -y ; 1\right),\end{equation}
\begin{equation} \label{3_bb3D} \tag{E4}\bb^\top(x,y)=\frac{25}{\pi}e^{-\frac{(x^2+y^2)}{200}}(-y;x;0), \quad \bb^\bot(x,y)=0.
\end{equation}
\begin{equation} \label{g1g2_1} \tag{E5}
g_1(x)=\frac{5}{16}\textbf{1}_{3.2\leq x\leq 4.2}(x-3.2)^2(4.2-x)^2, \quad  g_2(x)=-\frac{35}{16}\textbf{1}_{3.4\leq x\leq 4}(x-3.4)^2(4-x)^2.\end{equation}
\begin{equation} \label{g1g2_2} \tag{E6}
g_1(x)=\frac{125}{16}\textbf{1}_{3.7\leq x\leq 4.2}(x-3.7)^2(4.2-x)^2, \quad g_2(x)=\frac{125}{16}\textbf{1}_{3.4\leq x\leq 4}(x-3.4)^2(4-x)^2.\end{equation}

\bibliographystyle{abbrv}
\bibliography{biblio}

\begin{thebibliography}{10}

\bibitem{achenbach1}
J.~D. Achenbach.
\newblock {\em Wave propagation in elastic solids}.
\newblock North-Holland Series in Applied Mathematics and Mechanics. Elsevier,
  Amsterdam, 1975.

\bibitem{achenbach2}
J.~D. Achenbach and Y.~Xu.
\newblock Use of elastodynamic reciprocity to analyze point-load generated
  axisymmetric waves in a plate.
\newblock {\em Wave Motion}, 30(1):57--67, 1999.

\bibitem{akian1}
J.-L. Akian.
\newblock A proof of the completeness of lamb modes.
\newblock {\em Mathematical Methods in the Applied Sciences},
  45(3):1402–1419, 2021.

\bibitem{balogun1}
O.~Balogun, T.~W. Murray, and C.~Prada.
\newblock Simulation and measurement of the optical excitation of the s1 zero
  group velocity lamb wave resonance in plates.
\newblock {\em Journal of Applied Physics}, 102(6):064914, 2007.

\bibitem{baronian1}
V.~Baronian, A.~{Bonnet-Ben Dhia}, and E.~Lunéville.
\newblock Transparent boundary conditions for the harmonic diffraction problem
  in an elastic waveguide.
\newblock {\em Journal of Computational and Applied Mathematics},
  234(6):1945--1952, 2010.

\bibitem{baronian2}
V.~Baronian, L.~Bourgeois, B.~Chapuis, and A.~Recoquillay.
\newblock {Linear Sampling Method applied to Non Destructive Testing of an
  elastic waveguide: theory, numerics and experiments }.
\newblock {\em {Inverse Problems}}, 34(7):075006, 2018.

\bibitem{berenger1}
J.~P. Berenger.
\newblock A perfectly matched layer for the absorption of electromagnetic
  waves.
\newblock {\em Journal of Computational Physics}, 114(2):185--200, 1994.

\bibitem{besserer1}
H.~Besserer and P.~G. Malischewsky.
\newblock Mode series expansions at vertical boundaries in elastic waveguides.
\newblock {\em Wave Motion}, 39(1):41--59, 2004.

\bibitem{bhatia1}
H.~{Bhatia}, G.~{Norgard}, V.~{Pascucci}, and P.~{Bremer}.
\newblock The helmholtz-hodge decomposition - a survey.
\newblock {\em IEEE Transactions on Visualization and Computer Graphics},
  19(8):1386--1404, 2013.

\bibitem{bonnet3}
A.~S. Bonnet-Ben~Dhia, C.~Chambeyron, and G.~Legendre.
\newblock {On the use of perfectly matched layers in the presence of long or
  backward propagating guided elastic waves}.
\newblock {\em {Wave Motion}}, 51(2):266--283, 2014.

\bibitem{bonnetier1}
E.~Bonnetier, A.~Niclas, L.~Seppecher, and G.~Vial.
\newblock {Small defects reconstruction in waveguide from multifrequency
  one-side scattering data}.
\newblock {\em {Inverse Problems and Imaging}}, 16(2):417--450, 2022.

\bibitem{bourgeois1}
L.~Bourgeois and E.~Lun{\'e}ville.
\newblock {The linear sampling method in a waveguide: A modal formulation}.
\newblock {\em {Inverse Problems}}, 24(1), 2008.

\bibitem{chorin1}
A.~J. Chorin, J.~E. Marsden, and J.~E. Marsden.
\newblock {\em A mathematical introduction to fluid mechanics}, volume~3.
\newblock Springer, 1990.

\bibitem{fraser1}
W.~B. Fraser.
\newblock Orthogonality relation for the rayleigh–lamb modes of vibration of
  a plate.
\newblock {\em The Journal of the Acoustical Society of America},
  59(1):215--216, 1976.

\bibitem{grisvard1}
P.~Grisvard.
\newblock {\em Elliptic Problems in Nonsmooth Domains}.
\newblock Society for Industrial and Applied Mathematics, 2011.

\bibitem{hecht1}
F.~Hecht.
\newblock New development in freefem++.
\newblock {\em J. Numer. Math.}, 20(3-4):251--265, 2012.

\bibitem{kharrat1}
M.~Kharrat, M.~N. Ichchou, O.~Bareille, and W.~Zhou.
\newblock Pipeline inspection using a torsional guided-waves inspection system.
  part 1: Defect identification.
\newblock {\em International Journal of Applied Mechanics}, 6(4), 2014.

\bibitem{kirrmann1}
P.~Kirrmann.
\newblock On the completeness of lamb modes.
\newblock {\em Journal of Elasticity}, 37(1):39--69, 1994.

\bibitem{legrand1}
F.~Legrand, B.~Gérardin, J.~Laurent, C.~Prada, and A.~Aubry.
\newblock Negative refraction of lamb modes: A theoretical study.
\newblock {\em Physical Review B}, 98(21), 2018.

\bibitem{locker1}
H.~R. Locker and J.~Locker.
\newblock {\em Spectral theory of non-self-adjoint two-point differential
  operators}.
\newblock American Mathematical Soc., 2000.

\bibitem{mallat1}
S.~Mallat.
\newblock {\em A Wavelet Tour of Signal Processing, Chapter 1: Sparse
  Representations}.
\newblock Academic Press, Boston, third edition, 2009.

\bibitem{maupin1}
V.~Maupin.
\newblock {Surface waves across 2-D structures: a method based on coupled local
  modes}.
\newblock {\em Geophysical Journal International}, 93(1):173 -- 185, 1988.

\bibitem{merkulov1}
L.~G. Merkulov, S.~I. Rokhlin, and O.~P. Zobnin.
\newblock Calculation of the spectrum of wave numbers for lamb waves in a
  plate.
\newblock {\em The Soviet journal of nondestructive testing}, 6:369--373, 1970.

\bibitem{nazarov1}
S.~A. Nazarov.
\newblock The mandelstam energy radiation conditions and the umov--poynting
  vector in elastic waveguides.
\newblock {\em Journal of Mathematical Sciences}, 195(5):676--729, 2013.

\bibitem{pagneux1}
V.~Pagneux and A.~Maurel.
\newblock Lamb wave propagation in inhomogeneous elastic waveguides.
\newblock {\em Proceedings of the Royal Society of London. Series A:
  Mathematical, Physical and Engineering Sciences}, 458(2024):1913--1930, 2002.

\bibitem{pagneux2}
V.~Pagneux and A.~Maurel.
\newblock Lamb wave propagation in elastic waveguides with variable thickness.
\newblock {\em Proceedings of the Royal Society A: Mathematical, Physical and
  Engineering Sciences}, 462(2068):1315--1339, 2006.

\bibitem{ceccon1}
P.~C. Ribeiro, H.~F. {de Campos Velho}, and H.~Lopes.
\newblock Helmholtz–hodge decomposition and the analysis of 2d vector field
  ensembles.
\newblock {\em Computers and Graphics}, 55:80--96, 2016.

\bibitem{royer1}
D.~Royer, D.~P. Morgan, and E.~Dieulesaint.
\newblock {\em Elastic Waves in Solids I: Free and Guided Propagation}.
\newblock Advanced Texts in Physics. Springer Berlin Heidelberg, 1999.

\bibitem{sommerfeld1}
A.~Sommerfeld.
\newblock {\em Partial Differential Equations in Physics}.
\newblock Academic press, 1949.

\bibitem{stange1}
S.~Stange.
\newblock {\em Die Ausbreitung von Oberfl{\"a}chenwellen in Erdmodellen mit
  ebenen und zylindrischen vertikalen Strukturgrenzen}.
\newblock PhD thesis, Inst. f{\"u}r Geophysik der Univ., 1992.

\bibitem{treyssede1}
F.~Treyssède.
\newblock Three-dimensional modeling of elastic guided waves excited by
  arbitrary sources in viscoelastic multilayered plates.
\newblock {\em Wave Motion}, 52:33--53, 2015.

\bibitem{wiebel1}
A.~Wiebel, G.~Scheuermann, and C.~Garth.
\newblock Feature detection in vector fields using the helmholtz-hodge
  decomposition.
\newblock Master's thesis, University of Kaiserslautern, 2004.

\bibitem{wilcox1}
P.~Wilcox.
\newblock Modeling the excitation of lamb and sh waves by point and line
  sources.
\newblock {\em AIP Conference Proceedings}, 700(1):206--213, 2004.

\bibitem{willberg1}
C.~Willberg, S.~Duczek, J.~M. Vivar-Perez, and Z.~A.~B. Ahmad.
\newblock {Simulation Methods for Guided Wave-Based Structural Health
  Monitoring: A Review}.
\newblock {\em Applied Mechanics Reviews}, 67(1), 2015.
\newblock 010803.

\bibitem{williams1}
E.~G. Williams.
\newblock {\em Fourier Acoustics}.
\newblock Academic Press, 1999.

\end{thebibliography}

\end{document}